\documentclass[a4paper,11pt]{article}
\usepackage{amsmath, amsfonts, amscd, amssymb, amsthm, enumerate, xypic}

\def\bfB{\mathbf{B}}
\def\bfC{\mathbf{C}}

\newcommand{\Hom}{\operatorname{Hom}}
\newcommand{\Mat}{\operatorname{M}}
\newcommand{\Mats}{\operatorname{S}}
\newcommand{\Mata}{\operatorname{A}}

\newcommand{\id}{\operatorname{id}}
\newcommand{\GL}{\operatorname{GL}}
\newcommand{\Ker}{\operatorname{Ker}}
\newcommand{\Diag}{\operatorname{Diag}}
\newcommand{\Vect}{\operatorname{span}}
\newcommand{\im}{\operatorname{Im}}

\newcommand{\tr}{\operatorname{tr}}

\newcommand{\rk}{\operatorname{rk}}
\newcommand{\codim}{\operatorname{codim}}
\renewcommand{\setminus}{\smallsetminus}
\newcommand{\modu}{\operatorname{mod}}

\def\F{\mathbb{F}}
\def\K{\mathbb{K}}

\def\calA{\mathcal{A}}
\def\calB{\mathcal{B}}

\def\calD{\mathcal{D}}

\def\calH{\mathcal{H}}

\def\calK{\mathcal{K}}
\def\calL{\mathcal{L}}
\def\calM{\mathcal{M}}

\def\calR{\mathcal{R}}
\def\calS{\mathcal{S}}
\def\calT{\mathcal{T}}
\def\calU{\mathcal{U}}
\def\calV{\mathcal{V}}

\def\lcro{\mathopen{[\![}}
\def\rcro{\mathclose{]\!]}}

\theoremstyle{definition}
\newtheorem{Def}{Definition}[section]
\newtheorem{Not}[Def]{Notation}

\theoremstyle{plain}
\newtheorem{theo}{Theorem}[section]
\newtheorem{prop}[theo]{Proposition}
\newtheorem{cor}[theo]{Corollary}
\newtheorem{lemma}[theo]{Lemma}
\newtheorem{claim}{Claim}

\theoremstyle{plain}

\theoremstyle{remark}
\newtheorem{Rems}{Remarks}
\newtheorem{Rem}[Rems]{Remark}

\title{Range-compatible homomorphisms on matrix spaces}
\author{Cl\'ement de Seguins Pazzis\footnote{Universit\'e de Versailles Saint-Quentin-en-Yvelines, Laboratoire de Math\'ematiques
de Versailles, 45 avenue des Etats-Unis, 78035 Versailles cedex, France}
\footnote{e-mail address: dsp.prof@gmail.com}}

\begin{document}

\thispagestyle{plain}

\maketitle

\begin{abstract}
Let $\K$ be a (commutative) field, and $U$ and $V$ be finite-dimensional vector spaces over $\K$.
Let $\calS$ be a linear subspace of the space $\calL(U,V)$ of all linear operators from $U$ to $V$. A map $F : \calS \rightarrow V$
is called range-compatible when $F(s) \in \im s$ for all $s \in \calS$.
Obvious examples of such maps are the evaluation maps $s \mapsto s(x)$, with $x \in U$.

In this article, we classify all the range-compatible group homomorphisms on $\calS$ provided that
$\codim_{\calL(U,V)} \calS \leq 2\dim V-3$, unless
$\K$ has cardinality $2$ and $\codim_{\calL(U,V)} \calS = 2\dim V-3$.
Under those assumptions, it is shown that the linear range-compatible maps are the evaluation maps,
and the above upper-bound on the codimension of $\calS$ is optimal for this result to hold.

As an application, we obtain new sufficient conditions for the algebraic reflexivity of an operator space
and, with the above conditions on the codimension of $\calS$, we give an explicit description of the
range-restricting and range-preserving homomorphisms on $\calS$.
\end{abstract}



\tableofcontents

\section{Introduction}

\subsection{Main definitions and goals}\label{intro}

Throughout the article, $\K$ denotes an arbitrary (commutative) field, and $U$ and $V$ denote vector spaces over $\K$.
We denote by $\Mat_{n,p}(\K)$ the set of matrices with $n$ rows, $p$ columns and entries in $\K$.
The entries of matrices will always be denoted by small letters, e.g.\ the entry of the matrix $M$ at the $(i,j)$-spot
is denoted by $m_{i,j}$.

We denote by $\calL(U,V)$ the space of all linear operators from $U$ to $V$,
and by $\Hom(U,V)$ the space of all homomorphisms from $(U,+)$ to $(V,+)$.
Given a linear subspace $\calS$ of $\calL(U,V)$, the codimension of $\calS$ in $\calL(U,V)$ is denoted by
$\codim \calS$.

\begin{Def}
Let $U$ and $V$ be vector spaces, and $\calS$ be a linear subspace (or, more generally, a subgroup) of $\calL(U,V)$.
A \textbf{range-compatible map on $\calS$} is a map $F : \calS \rightarrow V$
that satisfies
$$\forall s \in \calS, \quad F(s) \in \im s.$$
\end{Def}

A similar definition is derived for maps from a linear subspace (or, more generally, a subgroup)
of $\Mat_{n,p}(\K)$ to $\K^n$ by using the canonical identification between $\Mat_{n,p}(\K)$ and $\calL(\K^p,\K^n)$.

\begin{Def}
Let $\calS$ be a linear subspace of $\calL(U,V)$.
A map $F : \calS \rightarrow V$ is called \textbf{local} when it is an evaluation map, i.e.\ when there is a vector
$x \in U$ such that
$$\forall s \in \calS, \quad F(s)=s(x).$$
One sees in that case that $F$ is linear and range-compatible.
\end{Def}

Note that the set of range-compatible homomorphisms on $\calS$ is a linear subspace of $\Hom(\calS,V)$,
and the one of local maps on $\calS$ is a linear subspace of it.

Again, we adopt a similar definition for maps between matrix spaces, so that, when $\calV$ is a linear subspace
(or, more generally, a subgroup) of
$\Mat_{n,p}(\K)$, a map $F : \calV \rightarrow \K^n$ is local if and only if there exists $X \in \K^p$
such that $F(M)=MX$ for all $M \in \calV$.

It has been noted by several authors that every range-compatible linear map on $\Mat_{n,p}(\K)$ is local
(e.g.\ this is implicit in \cite{Dieudonne}). In \cite{dSPclass}, this nice result is generalized as follows:

\begin{theo}[Lemma 8 of \cite{dSPclass}]\label{nonoptimaltheo}
Let $\calS$ be a linear subspace of $\Mat_{n,p}(\K)$ with $\codim \calS \leq n-2$.
Then, every range-compatible linear map on $\calS$ is local.
\end{theo}

Theorem \ref{nonoptimaltheo} was a major key in the generalization to arbitrary fields of an important theorem of Atkinson and Lloyd on the structure of large spaces of matrices with bounded rank (see \cite{AtkLloyd} for the
seminal result, and \cite{dSPclass} for its optimal generalization).
In \cite{dSPlargelinpres}, Theorem \ref{nonoptimaltheo} was also used to generalize Dieudonn\'e's theorem
on invertibility preservers \cite{Dieudonne} to a whole class of large subspaces of square matrices.

One of our current research problems is an optimal version of the main theorems of \cite{dSPlargelinpres}
for full-rank preservers on large spaces of non-square matrices. In investigating this problem, we have discovered
that the optimal upper bound is not $n-2$. Here it is:

\begin{Not}
If $\# \K>2$, we set $d_n(\K):=2n-3$. If $\# \K=2$, we set $d_n(\K):=2n-4$.
\end{Not}

\begin{theo}\label{maintheolin}
Let $\calS$ be a linear subspace of $\Mat_{n,p}(\K)$ with $\codim \calS \leq d_n(\K)$ and $p \geq 2$.
Then, every range-compatible linear map on $\calS$ is local.
\end{theo}

Let us prove right away that the upper-bound $2n-3$ is optimal for $\# \K>2$
(for fields with $2$ elements, we postpone the discussion until Section \ref{introsecondtheorem}).
Set
$$\calU:=\biggl\{\begin{bmatrix}
a & b \\
0 & a
\end{bmatrix} \mid (a,b)\in \K^2\biggr\} \subset \Mat_2(\K)$$
and
$$F : \begin{bmatrix}
a & b \\
0 & a
\end{bmatrix} \in \calU \longmapsto
\begin{bmatrix}
b  \\
0
\end{bmatrix}.$$
Then, $F$ is a range-compatible linear map as, for $M \in \calU$, either $m_{1,1} \neq 0$
and hence $M$ is invertible, or $m_{1,1}=0$ and $F(M)$ is the second column of $M$.
However, if there existed some $X\in \K^2$ such that $F(M)=MX$ for all $M \in \calU$, then in particular $X=F(I_2)=0$, contradicting the obvious fact that $F$ is non-zero.

More generally, given $n \geq 2$ and $p \geq 2$, one can consider the subspace $\calS \subset \Mat_{n,p}(\K)$
of all matrices of the form
$\begin{bmatrix}
A & B \\
[0]_{(n-2) \times 2} & C
\end{bmatrix}$, with $A \in \calU$, $B \in \Mat_{2,p-2}(\K)$ and $C \in \Mat_{n-2,p-2}(\K)$.
One sees that $\codim \calS=2n-2$ and that the linear map
$$M \in \calS \longmapsto \begin{bmatrix}
m_{1,2} \\
[0]_{(n-1) \times 1}
\end{bmatrix}$$
is range-compatible but not local. Thus, the upper bound $2n-3$ in Theorem \ref{maintheolin} is
optimal for all fields with more than $2$ elements and all integers $n \geq 2$ and $p \geq 2$.
For $n=1$ or $p=1$, it is easy to see that no upper bound on the codimension is necessary (this is obvious for $n=1$,
while for $p=1$ this follows from Proposition \ref{dimU=1} of Section \ref{dim1section}).

In this article, we shall not limit ourselves to the study of linear range-compatible maps,
rather we will enlarge the discussion to encompass all range-compatible \emph{group homomorphisms}.
The motivation for doing so is twofold:
firstly, range-compatible maps are of no interest if we do not add an additional algebraic property,
and additivity seems to be the minimal algebraic requirement if we want to find any meaningful result;
secondly, because the fundamental theorem of projective geometry plays a large part in the study of
full-rank preserving maps \cite{Dieudonne,dSPlargelinpres}, we actually need a classification of
all range-compatible \emph{semi-linear} maps on large operator spaces. In this prospect however, we stumble across
the difficulty that the set of all semi-linear maps between two vector spaces is not closed under addition
in general (as there may be different field automorphisms attached to those maps).
The enlargement of our framework to group homomorphisms is a natural way of avoiding that difficulty.

For range-compatible homomorphisms, Theorem \ref{nonoptimaltheo}
has a simple extension:

\begin{theo}[First Classification Theorem]\label{classtheo1}
Let $U$ and $V$ be finite-dimensional vector spaces over $\K$.
Let $\calS$ be a linear subspace of $\calL(U,V)$ with $\codim \calS \leq \dim V-2$.
Then, every range-compatible homomorphism on $\calS$ is local.
\end{theo}

The reason to coin Theorem \ref{classtheo1} as the First Classification Theorem lies
in the fact that, although the upper bound $\dim V-2$ is not optimal for linear maps,
it is optimal for group homomorphisms. To see this, consider a group endomorphism $\varphi$ of $(\K,+)$,
denote by $\calS$ the space of all $n \times p$ matrices of the form
$$\begin{bmatrix}
? & [?]_{1 \times (p-1)} \\
[0]_{(n-1) \times 1} & [?]_{(n-1) \times (p-1)}
\end{bmatrix},$$
and consider the map
$$F : M \in \calS \longmapsto \begin{bmatrix}
\varphi(m_{1,1}) \\
[0]_{(n-1) \times 1}
\end{bmatrix}.$$
We see that $F$ is a range-compatible homomorphism on $\calS$
(if $m_{1,1} \neq 0$, then $F(M)$ is a scalar multiple of the first column of $M$, otherwise
$F(M)=0$). However, if $\varphi$ is non-linear, then $F$ is also non-linear whence $F$ cannot be local.
In particular, whenever $\K$ is non-prime, it has non-linear group endomorphisms\footnote{
E.g., one denotes by $\K_0$ the prime subfield of $\K$ and one takes a non-zero linear form on the $\K_0$-vector space $\K$.},
yielding a linear subspace of $\Mat_{n,p}(\K)$ with codimension $n-1$ on which
at least one range-compatible homomorphism is non-local. If $\K$ is prime, then
every group homomorphism between vector spaces over $\K$ is linear, and in this case the extension of Theorem \ref{nonoptimaltheo} to group homomorphisms is trivial.

Only slight modifications of the proof of Lemma 8 of \cite{dSPclass} are needed to prove the First Classification
Theorem. Rather than explain those modifications, we will give a more efficient proof of Theorem \ref{classtheo1} that
uses a completely different strategy than the one of Lemma 8 of \cite{dSPclass}
and highlights the new techniques that are developed in this article.
Those techniques will actually help us obtain a far more comprehensive and difficult result
that describes all the range-compatible homomorphisms provided that $\codim \calS \leq 2\dim V-3$
and the underlying field has more than $2$ elements. For such fields, Theorems \ref{maintheolin} and \ref{classtheo1} will appear
as obvious corollaries. As this ultimate theorem involves quite a few ``wild" cases, we postpone
its statement until Section \ref{introsecondtheorem}.

Recently, an additional motivation for studying range-compatible linear maps came from the discovery
of their profound relationship with the concept of algebraic reflexivity for operator spaces.
The reflexive closure of a linear subspace $\calS$ of $\calL(U,V)$
is defined as the space $\calR(\calS)$ of all linear maps $f$ for which $\forall x \in U, \; f(x) \in \calS x$
(note that $\calS \subset \calR(\calS)$ and that $\calR(\calR(\calS))=\calR(\calS)$).
The space $\calS$ is called \textbf{(algebraically) reflexive} when $\calR(\calS)=\calS$.
For $x \in U$, consider the linear operator $\hat{x} : f \in \calS \mapsto f(x) \in V$, so that
$$\widehat{\calS}:=\bigl\{\hat{x} \mid x \in U\bigr\}$$
is a linear subspace of $\calL(\calS,V)$.
Let $g \in \calR(\calS)$. For $x \in U$, we see that $\hat{x}=0 \Rightarrow g(x)=0$.
Thus, we obtain a linear map $G : \widehat{\calS} \rightarrow V$ such that $\forall x \in U, \; g(x)=G(\widehat{x})$.
The assumption $g \in \calR(\calS)$ means that $g(x) \in \im \widehat{x}$ for all $x \in U$, which yields that
$G$ is range-compatible. Note that $G$ is local if and only if there exists $f \in \calS$
such that $\forall x \in U, \; G(\widehat{x})=\widehat{x}(f)$, that is $f=g$, whence
$G$ is local if and only if $g \in \calS$.

Conversely, let $H : \widehat{\calS} \rightarrow V$ be a linear range-compatible map.
Then, $h : x \in U \mapsto H(\widehat{x})$ is linear and belongs to $\calR(\calS)$.
Again, $H$ is local if and only if $h$ belongs to $\calS$. Thus, we have proved the following link between the
localness of range-compatible linear maps and algebraic reflexivity:

\begin{prop}
Let $\calS$ be a linear subspace of $\calL(U,V)$. Then, $\calS$ is algebraically reflexive if and only if
every range-compatible linear map on $\widehat{\calS}$ is local.
\end{prop}

Thus, finding sufficient conditions for all the range-compatible linear maps on $\widehat{\calS}$ to be local
yields sufficient conditions for $\calS$ to be algebraically reflexive.

\begin{Rem}
Implicit in the above construction is an isomorphism between $\calR(\calS)$ and the vector space
of all range-compatible linear maps from $\widehat{\calS}$ to $V$.
\end{Rem}

Now, we shall say that an operator space $\calS \subset \calL(U,V)$
is \textbf{reduced} when no vector of $U$ belongs to all the kernels of the operators in $\calS$
and when $V$ is the sum of all the ranges of the operators in $\calS$. It is clear that
algebraic reflexivity needs only be studied for reduced spaces.
Given such a reduced space $\calS$, we see that $\widehat{\calS}$ is a linear subspace of $\calL(\calS,V)$
with dimension $\dim U$. Therefore, Theorem \ref{maintheolin}
yields the following corollary on algebraic reflexivity:

\begin{theo}
Let $U$ and $V$ be finite-dimensional vector spaces, and $\calS$ be a reduced linear subspace of $\calL(U,V)$.
Assume that
\begin{itemize}
\item $\dim U \geq \dim \calS \dim V-2\dim V+3$ if $\# \K>2$;
\item $\dim U \geq \dim \calS \dim V-2\dim V+4$ if $\# \K=2$.
\end{itemize}
Then, $\calS$ is algebraically reflexive.
\end{theo}

Again, these conditions are optimal assuming $\dim V \geq 2$.

\subsection{Additional definition and notation}

When we talk of hyperplanes, we will always mean linear hyperplanes unless specified otherwise.

We denote by $\Mat_n(\K)$ the algebra of $n \times n$ square matrices with entries in $\K$, by
$\GL_n(\K)$ its group of invertible elements, by $\Mats_n(\K)$ its subspace of symmetric matrices, and by
$\Mata_n(\K)$ its subspace of alternating matrices (i.e. skew-symmetric matrices with all diagonal entries zero).
Given non-negative integers $m,n,p,q$ and respective subsets $\calA$ and $\calB$ of $\Mat_{m,p}(\K)$ and $\Mat_{n,q}(\K)$,
one sets
$$\calA \vee \calB:=\biggl\{\begin{bmatrix}
A & C \\
[0]_{n \times p} & B
\end{bmatrix} \mid A \in \calA,\; B \in \calB,\; C \in \Mat_{m,q}(\K)\biggr\} \subset \Mat_{m+n,p+q}(\K).$$

Given non-negative integers $n,p,q$ and respective subsets $\calA$ and $\calB$ of $\Mat_{n,p}(\K)$ and $\Mat_{n,q}(\K)$,
one sets
$$\calA \coprod \calB :=\biggl\{\begin{bmatrix}
A & B
\end{bmatrix} \mid A \in \calA, \; B \in \calB\biggr\}.$$

The rank of $M \in \Mat_{n,p}(\K)$ is denoted by $\rk M$, and the trace of a square matrix $M$
is denoted by $\tr(M)$. A similar notation is used for the trace of an endomorphism of a finite-dimensional vector space.

Given integers $i \in \lcro 1,n\rcro$ and $j \in \lcro 1,p\rcro$, we denote by $E_{i,j}$ the matrix of $\Mat_{n,p}(\K)$
with all entries zero except the one at the $(i,j)$-spot, which equals $1$. When we use this notation, the
integers $n$ and $p$ will always be obvious from the context.

We make the group $\GL_n(\K) \times \GL_p(\K)$ act on the set of linear subspaces of $\Mat_{n,p}(\K)$
by
$$(P,Q).\calV:=P\,\calV\,Q^{-1}.$$
Two linear subspaces of the same orbit will be called \textbf{equivalent}
(this means that they represent, in a change of bases, the same set of linear transformations from a $p$-dimensional vector space to an $n$-dimensional vector space).

When $U$ and $V$ are finite-dimensional, we consider the bilinear form
$$(u,v) \in \calL(U,V) \times \calL(V,U) \longmapsto \tr(v \circ u).$$
It is non-degenerate on both sides. In the rest of the text, orthogonality will always refer to this form, to the effect that,
given a subset $\calS$ of $\calL(U,V)$, one has
$$\calS^\bot:=\bigl\{v \in \calL(V,U) : \; \forall u \in \calL(U,V), \; \tr(v \circ u)=0\bigr\}.$$
Recall that $(\calS^\bot)^\bot=\calS$ whenever $\calS$ is a linear subspace of $\calL(U,V)$.

\begin{Def}
Let $U$ and $V$ be finite-dimensional vector spaces (with $n:=\dim V$ and $p:=\dim U$), $\calS$ be a linear subspace of $\calL(U,V)$, and
$F : \calS \rightarrow V$ be a range-compatible map.
Given bases $\bfB$ and $\bfC$, respectively, of $U$ and $V$,
we consider the space $\calM:=\{M_{\bfB,\bfC}(s) \mid s \in \calS\}$ and
define $F_{\bfB,\bfC} : \calM \rightarrow \K^n$ as the sole map which makes the following diagram commutative:
$$\xymatrix{
\calS \ar[rr]^{s \mapsto M_{\bfB,\bfC}(s)}_{\simeq} \ar[d]_{F} & & \calM \ar[d]^{F_{\bfB,\bfC}} \\
V \ar[rr]_{y \mapsto M_\bfC(y)}^{\simeq} & &  \K^n.
}$$
We say that $F_{\bfB,\bfC}$ represents $F$ in the bases $\bfB$ and $\bfC$.
\end{Def}

Note that changing the chosen bases amounts to replacing $F_{\bfB,\bfC}$ with a map of the form
$M \mapsto PF_{\bfB,\bfC}(P^{-1}MQ)$ for some $(P,Q)\in \GL_n(\K) \times \GL_p(\K)$.

\subsection{The Second Classification Theorem}\label{introsecondtheorem}

Before we can state our main classification theorem for range-compatible homomorphisms,
we need to discuss additional examples of such maps, beyond the local ones.

Our first example is a generalization of the one given after the statement of Theorem \ref{classtheo1}.
Let $\calS$ be a linear subspace of $\calL(U,V)$ for which we have a vector
$x \in U$ that satisfies $\dim \calS x=1$, and let $\alpha : \calS x \rightarrow \calS x$
be a group homomorphism.
Then, the homomorphism $s \mapsto \alpha(s(x))$ is range-compatible: indeed, given $s \in \calS$,
either $s(x) \neq 0$ and hence $\calS x=\K s(x) \subset \im s$, to the effect that $\alpha(s(x)) \in \im s$,
or $s(x)=0$ whence $\alpha(s(x))=0 \in \im s$.
On the other hand, $s \mapsto \alpha(s(x))$ is linear if and only if $\alpha$ is linear,
in which case we have a scalar $\lambda$ such that $\forall s \in \calS, \; \alpha(s(x))=\lambda\, s(x)=s(\lambda x)$,
whence $s \mapsto \alpha(s(x))$ is local.

In terms of matrices, this example can be restated as follows:
let $\calS$ be a linear subspace of $\Mat_1(\K) \vee \Mat_{n-1,p-1}(\K)$ that is not included in $\{0\} \coprod \Mat_{n,p-1}(\K)$,
and let $\alpha : \K \rightarrow \K$ be a group endomorphism.
Then,
$$M \in \calS \mapsto \begin{bmatrix}
\alpha(m_{1,1}) \\
[0]_{(n-1)\times 1}
\end{bmatrix}$$
is a range-compatible homomorphism, and it is local if and only if $\alpha$ is linear.

Now, we turn to less obvious examples of non-linear range-compatible homomorphisms:
assume that $\K$ has characteristic $2$; given a vector space $U$ over $\K$, denote by $U^{/2}$
the vector space over $\K$ with the same underlying abelian group as $U$ but with the new scalar multiplication ``$\bullet$"
defined from the former one ``$\cdot$" as
$$\lambda\bullet x:=\lambda^2\cdot x.$$
\begin{Def}
Given vector spaces $U$ and $V$ over $\K$,
a map $\alpha : U \rightarrow V$ is called \textbf{root-linear} when it is linear as a map from $U^{/2}$ to $V$,
i.e.\ $\alpha$ is a homomorphism such that
$$\forall (\lambda,x)\in \K\times U, \quad \alpha(\lambda^2 x)=\lambda\, \alpha(x).$$
\end{Def}
Given a root-linear form $\alpha$ on $\K$, we will show in Section \ref{symmetricsection} (this is a non-trivial result!) that the homomorphism
$$M \in \Mats_n(\K) \longmapsto \begin{bmatrix}
\alpha(m_{1,1}) \\
\alpha(m_{2,2}) \\
 \vdots \\
 \alpha(m_{n,n})
\end{bmatrix}$$
obtained by extracting the diagonal and applying $\alpha$ entry-wise
is range-compatible and that it is non-local if $\alpha \neq 0$.
For example, if $\K$ is a perfect field of characteristic $2$, we define $\sqrt{x}$ as the sole square root in $\K$ of the element $x \in \K$,
and we consider the inverse $\alpha : x \mapsto \sqrt{x}$ of the Frobenius automorphism of $\K$.
Note that over $\F_2$, root-linear maps are linear maps, and in particular the only non-zero root-linear form on $\F_2$
is the identity. In particular, the map
$$\begin{bmatrix}
a & b \\
b & c
\end{bmatrix} \in \Mats_2(\F_2) \longmapsto \begin{bmatrix}
a \\
c
\end{bmatrix}$$
is range-compatible but non-local. As in Section \ref{intro},
we extend this example by taking integers $n \geq 2$ and $p \geq 2$ and by considering the space $\calS=\Mats_2(\F_2) \vee \Mat_{n-2,p-2}(\F_2)$
and the map
$$F : M \in \calS \mapsto \begin{bmatrix}
m_{1,1} \\
m_{2,2} \\
[0]_{(n-2) \times 1}
\end{bmatrix};$$
one checks that $\codim \calS=2n-3$ and that $F$ is a non-local range-compatible linear map.
Therefore, the upper bound $2n-4$ in Theorem \ref{maintheolin} is optimal for $\K=\F_2$ and all integers $n \geq 2$ and $p \geq 2$.

\begin{Def}
Let $\calS$ be a linear subspace of $\calL(U,V)$. \\
We say that $\calS$ has \textbf{Type 1} when there is a vector $x \in U$ such that $\dim \calS x=1$. \\
We say that $\calS$ has \textbf{Type 2} when $\K$ has characteristic $2$ and
$\calS$ is represented, in well-chosen bases of $U$ and
$V$, by the matrix space $\Mats_2(\K) \vee \Mat_{n-2,p-2}(\K)$, where $p:=\dim U$ and $n:=\dim V$. \\
We say that $\calS$ has \textbf{Type 3} when $\K$ has characteristic $2$ and $\calS$ is represented, in well-chosen bases of $U$ and $V$, by the matrix space $\Mats_3(\K) \coprod \Mat_{3,p-3}(\K)$, where $p:=\dim U$.
\end{Def}

Note that all those types are incompatible: if $\calS$ has Type 2 or Type 3, then we see that $\dim \calS x \geq 2$ for all $x \in U \setminus \{0\}$,
which rules out Type $1$;
if $\calS$ has Type $2$, then there is a vector $x \in U$ such that $\dim \calS x=2$, whereas, if $\calS$ has Type $3$
we see that $\dim \calS x =3$ for all non-zero vectors $x \in U$.

We are now ready to state the Second Classification Theorem for range-compatible homomorphisms
from an operator space with small codimension over a field with more than $2$ elements:

\begin{theo}[Second Classification Theorem]\label{classtheo2}
Let $\K$ be a field with more than $2$ elements.
Let $U$ and $V$ be finite-dimensional vector spaces over $\K$, with respective dimensions
$p \geq 1$ and $n \geq 2$. Let $\calS$ be a linear subspace of $\calL(U,V)$ with $\codim \calS \leq 2n-3$.
Then:
\begin{enumerate}[(a)]
\item Every range-compatible linear map on $\calS$ is local.
\item If $\calS$ has none of Types 1 to 3, then every range-compatible homomorphism on $\calS$
is local.
\item If $\calS$ has Type 1, then, given a vector $x \in U$ such that $\dim \calS x=1$,
the range-compatible homomorphisms on $\calS$ are the sums of the local maps and the maps
of the form $s \mapsto \alpha(s(x))$ where $\alpha : \calS x \rightarrow \calS x$ is a group homomorphism.

\item If $\calS$ has Type 2, then, given bases of $U$ and $V$ in which $\calS$ is represented by the matrix
space $\Mats_2(\K) \vee \Mat_{n-2,p-2}(\K)$, the range-compatible homomorphisms on $\calS$ are represented by the sums of
the local maps and the maps of the form $M \mapsto \begin{bmatrix}
\alpha(m_{1,1}) \\
\alpha(m_{2,2}) \\
[0]_{(n-2) \times 1}
\end{bmatrix}$ where $\alpha$ is a root-linear form on $\K$.

\item If $\calS$ has Type 3, then, given bases of $U$ and $V$ in which $\calS$ is represented by the matrix
space $\Mats_3(\K) \coprod \Mat_{3,p-3}(\K)$, the range-compatible homomorphisms on $\calS$ are represented by the sums of
the local maps and the maps of the form $M \mapsto \begin{bmatrix}
\alpha(m_{1,1}) \\
\alpha(m_{2,2}) \\
\alpha(m_{3,3})
\end{bmatrix}$ where $\alpha$ is a root-linear form on $\K$.
\end{enumerate}
\end{theo}

An important step in the proof of the Second Classification Theorem
is the description of all range-compatible homomorphisms on the space of all symmetric matrices. As this result is interesting in itself, we highlight it here:

\begin{theo}[Range-compatible homomorphisms on full spaces of symmetric matrices]\label{symmetrictheo}
Assume that $n \geq 2$.
If $\K$ has characteristic not $2$, then every
range-compatible homomorphism on $\Mats_n(\K)$ is local. \\
If $\K$ has characteristic $2$, then the group of all range-compatible homomorphisms on $\Mats_n(\K)$
is generated by the local maps together with the maps of the form
$$M \longmapsto \begin{bmatrix}
\alpha(m_{1,1}) \\
\alpha(m_{2,2}) \\
 \vdots \\
\alpha(m_{n,n})
\end{bmatrix},$$
where $\alpha$ is a root-linear form on $\K$.
\end{theo}

\subsection{Strategy of proof, and structure of the article}

So far, the best known result on range-compatible homomorphisms on large operator spaces
was Theorem \ref{nonoptimaltheo}. The approach featured in \cite{dSPclass} seems to have delivered all its potential,
and brand new methods are unavoidable to prove both Theorem \ref{maintheolin} and the Second Classification Theorem. Thus, all the techniques in this article are new except the ones we shall use to prove point (c) of Theorem \ref{classtheo2}, where we borrow the line of reasoning from \cite{dSPclass}.

Let us explain the key ingredients of the new method.
The most important new insight is the \emph{projection technique}
(Section \ref{projectiontechnique}): given a non-zero vector $y$,
one sees that a range-compatible homomorphism $F$ on $\calS \subset \calL(U,V)$
induces a range-compatible homomorphism on the space $\calS \modu y$
of all linear maps $\pi \circ s$, with $s \in \calS$, where $\pi : V \twoheadrightarrow V/\K y$ is the standard projection.
With this technique, we can obtain results on range-compatible homomorphisms
by using an induction on the dimension of the target space $V$: this is the exact opposite of the proof method of
Theorem \ref{nonoptimaltheo} (from \cite{dSPclass}), which was based on an induction on
the dimension of the source space $U$!
This projection technique will help us obtain Theorems \ref{maintheolin}, \ref{classtheo1} and \ref{classtheo2}.
In each case, it is necessary to choose a sufficient number of \emph{adapted} vectors $y \in V$.
In short, a vector $y$ of $V$ can be considered as adapted to $\calS$ when the codimension of
$\calS \modu y$ in $\calL(U,V/\K y)$ is small enough (with respect to the theorem we wish to prove),
so that we can apply an induction hypothesis to $\calS \modu y$. If we have enough linearly independent
adapted vectors, then the range-compatible homomorphisms on $\calS$ can be easily computed
(in the case of the First Classification Theorem, two adapted vectors are sufficient, whereas three are needed for to yield
Theorem \ref{maintheolin} and point (b) of Theorem \ref{classtheo2}).

Another technique that will be used frequently is the \emph{splitting} of matrix spaces.
When a matrix space $\calS$ splits as $\calA \coprod \calB$ for some matrix spaces $\calA$ and $\calB$,
the range-compatible homomorphisms on $\calS$ are easily deduced from those on $\calA$ and $\calB$
(see Lemma \ref{splittinglemma}).

The rest of the article is laid out as follows.
Section \ref{maintechniques} is devoted to the basic techniques that are used throughout the article:
the first one is Theorem \ref{classtheo1} when the source space has dimension $1$, which is basically a reformulation
of a classical characterization of scalar multiples of the identity (Lemma \ref{carachom}); then, we shall quickly explain the embedding and splitting techniques, and subsequently use them to compute the range-compatible homomorphisms in specific situations, most notably in the one
of the space of all linear maps from $U$ to $V$ (Section \ref{totalspacesection}); the next paragraph
(Section \ref{projectiontechnique}) deals with the projection technique. In the last two paragraphs of Section \ref{maintechniques},
we use those techniques to prove Theorem \ref{classtheo1} (with a completely different line of reasoning than in \cite{dSPclass}), and we derive
point (c) of Theorem \ref{classtheo2} from Theorem \ref{classtheo1} by following the method of \cite{dSPclass}.

The next section is devoted to the proof of Theorem \ref{symmetrictheo}, i.e.\ the determination of all range-compatible homomorphisms on the space
of all $n \times n$ symmetric matrices (with $n \geq 2$).
Points (d) and (e) of Theorem \ref{classtheo2} follow as easy corollaries.

The next section is devoted to the proof of Theorem \ref{maintheolin}.
In it, we introduce the notion of an $\calS$-adapted vector (in the context of Theorems \ref{maintheolin} and \ref{classtheo2}).
Then, we demonstrate (Section \ref{adaptedvectorssection}) that a triple of linearly independent $\calS$-adapted vectors
exists except in a very specific situation with $n=3$,
and we prove a key result that points out how important it is to have such a triple (Lemma \ref{3vectorslemma}).

Equipped with those results, and after a quick examination of the case $n=2$ (Section \ref{n=2section}),
we prove Theorem \ref{maintheolin} by induction on $n$, first for fields with more than $2$ elements
(Section \ref{linearmorethan2}), and then for fields with two elements (Section \ref{linear2}).

At that point of the article, the only statement of Theorem \ref{classtheo2} that will remain unproven
is point (b). Section \ref{nonlinearsection} is devoted to its proof:
again, this is an inductive proof, in which we will assume that the space $\calS$ has neither Type 1 nor Type 3 and that there is a non-linear range-compatible homomorphism on $\calS$, and we will prove that $\calS$ must have Type 2.
A major key in the proof lies in the analysis of the orthogonal space $\calS^\bot$:
under the above assumptions, we shall prove that all the operators in the dual operator space $\widehat{\calS^\bot}$ have rank less than $3$, and
the classification of spaces of matrices with rank less than $3$ will help us have a better grasp of the actual structure of $\calS$.

In the final section, we will apply Theorem \ref{classtheo2} to obtain a classification of all
range-restricting and range-preserving homomorphisms on linear subspaces of $\Mat_{n,p}(\K)$
with codimension at most $d_n(\K)$. These notions are defined in Section \ref{rangerestrictingsection}.
The ultimate results classify the range-preserving semi-linear maps under the same condition on the codimension of
the matrix space under consideration: they are needed in the study of full-rank preserving linear maps on
large spaces of rectangular matrices.

\section{Main techniques}\label{maintechniques}

\subsection{The case $\dim U=1$}\label{dim1section}

The following lemma is known but we reprove it for the sake of completeness.
Carefully reformulated, it yields the case $\dim U=1$ in the First Classification Theorem.

\begin{lemma}\label{carachom}
Let $f : V \rightarrow V$ be a group homomorphism such that
$f(x) \in \K x$ for all $x \in V$. \\
If $\dim V \neq 1$ or $f$ is linear, then there exists $\lambda \in \K$ such that $f : x \mapsto \lambda\,x$.
\end{lemma}

\begin{proof}
If $\dim V=1$ and $f$ is linear, or if $\dim V=0$, then the result is straightforward.

Assume now that $\dim V> 1$.
For every $x \in V \setminus \{0\}$, there is a unique $\lambda_x \in \K$ such that $f(x)=\lambda_x\,x$. \\
Let $y$ and $z$ be linearly independent vectors of $V$. Then,
$\lambda_{y+z}(y+z)=f(y+z)=f(y)+f(z)=\lambda_y\,y+\lambda_z\,z$, which yields $\lambda_y=\lambda_{y+z}=\lambda_z$. \\
Given two linearly dependent vectors $y$ and $z$ of $V \setminus \{0\}$, one may choose $x \in V \setminus \K y$.
Applying the previous result to both pairs $(x,y)$ and $(x,z)$ leads to $\lambda_y=\lambda_x=\lambda_z$.
We deduce that the map $x \in V \setminus \{0\} \mapsto \lambda_x$ has a sole value $\lambda$, which yields
$\forall x \in V, \; f(x)=\lambda\,x$ since $f(0)=0$.
\end{proof}

\begin{cor}\label{dimU=1}
Assume that $\dim U=1$, and let $\calS$ be a linear subspace of $\calL(U,V)$.
Let $F : \calS \rightarrow V$ be a range-compatible map. Assume that $\dim \calS\neq 1$ or $F$ is linear. Then, $F$ is local.
\end{cor}

\begin{proof}
Choosing a non-zero vector $x_1$ of $U$, we see that $s \mapsto s(x_1)$ defines a linear isomorphism from $\calS$
to a linear subspace $V_0$ of $V$, yielding a group homomorphism $f : V_0 \rightarrow V$ such that
$f(s(x_1))=F(s)$ for all $s \in \calS$, and $f$ is linear if and only if $F$ is linear.
Since $F$ is range-compatible, we deduce that $f(x) \in \K x$ for all
$x \in V_0$. In particular, $f$ is an endomorphism of $V_0$.
As $\dim V_0=\dim \calS$, Lemma \ref{carachom}
yields some $\lambda \in \K$ such that $F(s)=\lambda s(x_1)=s(\lambda x_1)$ for all $s \in \calS$.
Thus, $F$ is local.
\end{proof}

\subsection{Embedding and splitting techniques}

The following lemma is obvious but will be very useful in our proofs.

\begin{lemma}[Embedding Lemma]\label{extensionlemma}
Let $\calS$ be a linear subspace of $\Mat_{n,p}(\K)$, and let $n'$ be a non-negative integer.
Consider the space $\calS' \subset \Mat_{n+n',p}(\K)$ of all matrices of the form $\begin{bmatrix}
M \\
[0]_{n' \times p}
\end{bmatrix}$ with $M \in \calS$, and let $F' : \calS' \rightarrow \K^{n+n'}$ be a range-compatible homomorphism. \\
Then, there is a range-compatible homomorphism $F : \calS \rightarrow \K^n$ such that
$$\forall M \in \calS, \; F'\Biggl(\begin{bmatrix}
M \\
[0]_{n' \times p}
\end{bmatrix}\Biggr)=\begin{bmatrix}
F(M) \\
[0]_{n' \times p}
\end{bmatrix}.$$
\end{lemma}

The next lemma deals with the situation of a ``split" matrix space:

\begin{lemma}[Splitting Lemma]\label{splittinglemma}
Let $n,p,q$ be non-negative integers, and $\calA$ and $\calB$ be linear subspaces, respectively, of $\Mat_{n,p}(\K)$ and $\Mat_{n,q}(\K)$.

Given maps $f : \calA \rightarrow \K^n$ and $g : \calB \rightarrow \K^n$,
set
$$f \coprod g : \begin{bmatrix}
A & B
\end{bmatrix} \in \calA \coprod \calB \longmapsto f(A)+g(B).$$
Then:
\begin{enumerate}[(a)]
\item The homomorphisms from $\calA \coprod \calB$ to $\K^n$
are the maps of the form $f \coprod g$, where $f\in \Hom(\calA,\K^n)$ and $g \in \Hom(\calB,\K^n)$.
Moreover, every homomorphism from $\calA \coprod \calB$ to $\K^n$ may be expressed in a unique fashion as $f \coprod g$.
\item The linear maps from $\calA \coprod \calB$ to $\K^n$
are the maps of the form $f \coprod g$, where $f\in \calL(\calA,\K^n)$ and $g \in \calL(\calB,\K^n)$.
\item Given $f \in \Hom(\calA,\K^n)$ and $g \in \Hom(\calB,\K^n)$, the map
$f \coprod g$ is range-compatible if and only if $f$ and $g$ are range-compatible.
\item Given $f \in \Hom(\calA,\K^n)$ and $g \in \Hom(\calB,\K^n)$, the map
$f \coprod g$ is local if and only if $f$ and $g$ are local.
\end{enumerate}
\end{lemma}

\begin{proof}
The first statement is obvious. More precisely, given
$F \in \Hom(\calA \coprod \calB, \K^n)$, we define $f : A \in \calA \mapsto F\Bigl(\begin{bmatrix}
A & [0]_{n \times q}
\end{bmatrix}\Bigr)$ and
$g :
B \in \calB \mapsto F\Bigl(\begin{bmatrix}
[0]_{n \times p} & B
\end{bmatrix}\Bigr)$, and we see that $F=f \coprod g$.
It is clear that $f \coprod g$ is linear if and only if $f$ and $g$ are linear, which yields point (b).

If $f$ and $g$ are range-compatible, then it is obvious that $F$ is also range-compatible.
The converse is also clear from the above definition of $f$ and $g$.

If $F$ is local, we have a vector $X \in \K^{p+q}$ such that $F(M)=MX$ for all $M \in \calA \coprod \calB$;
splitting $X=\begin{bmatrix}
X_1 \\
X_2
\end{bmatrix}$ with $X_1 \in \K^p$ and $X_2 \in \K^q$, we see that $f(A)=AX_1$ and $g(B)=BX_2$ for all $(A,B) \in \calA \times \calB$,
whence $f$ and $g$ are both local. The converse statement is straightforward.
\end{proof}

\subsection{The case of the full space of linear maps}\label{totalspacesection}

Using the splitting technique and the case $\dim U=1$, we can give a quick proof of
the most basic result on range-compatible homomorphisms:

\begin{prop}\label{espacetotal}
Let $U$ and $V$ be finite-dimensional vector spaces.
\begin{enumerate}[(a)]
\item Every range-compatible linear map on $\calL(U,V)$ is local.
\item If $\dim V>1$, then every range-compatible homomorphism on $\calL(U,V)$ is local.
\end{enumerate}
\end{prop}

\begin{proof}
Setting $p:=\dim U$ and $n:=\dim V$, it suffices to prove the matrix versions of the statements, that is:
\begin{enumerate}[(a)]
\item Every range-compatible linear map on $\Mat_{n,p}(\K)$ is local.
\item If $n>1$, then every range-compatible homomorphism on $\Mat_{n,p}(\K)$ is local.
\end{enumerate}
The case $p=1$ is a simple reformulation of Lemma \ref{dimU=1} into the matrix setting.
From there, both results are obtained by induction on $p$ by noting that $\Mat_{n,p}(\K)=\Mat_{n,p-1}(\K) \coprod \Mat_{n,1}(\K)$
and by using the Splitting Lemma.
\end{proof}

\subsection{The projection technique}\label{projectiontechnique}

Now, we turn to a more profound technique that will help us perform inductive proofs.

\begin{lemma}[Projection Lemma]
Let $\calS$ be a linear subspace of $\calL(U,V)$, and $V_0$ be a linear subspace of $V$.
Let $F : \calS \rightarrow V$ be a range-compatible homomorphism.
Denote by $\pi : V \twoheadrightarrow V/V_0$ the canonical projection, and by
$\calS \modu V_0$ the space of all linear maps $\pi \circ s$ with $s \in \calS$.
Then, there is a unique range-compatible homomorphism $F \modu V_0 : \calS \modu V_0 \rightarrow V/V_0$
such that $\forall s \in \calS, \; (F\modu V_0)(\pi \circ s)=\pi(F(s))$,
i.e.\ the following diagram is commutative:
$$\xymatrix{
\calS \ar[rr]^F \ar[d]_{s \mapsto \pi \circ s} & & V \ar[d]^\pi \\
\calS \modu V_0 \ar[rr]_{F \modu V_0} & & V/V_0.
}$$
Moreover, $F \modu V_0$ is linear whenever $F$ is linear. \\
In particular, given a non-zero vector $y \in V$, one denotes by $F \modu y$ the projected map $F \modu \K y$, and by $\calS \modu y$
the operator space $\calS \modu \K y$.
\end{lemma}

\begin{proof}
It suffices to apply the factorization theorem for group homomorphisms (respectively, for linear maps) as,
for any $s \in \calS$ satisfying $\pi \circ s=0$, one finds $\im s \subset V_0$ and hence $\pi(F(s))=0$ since $F(s) \in \im s$.
\end{proof}

In terms of matrices, the special case when $V_0$ is a linear hyperplane of $V$ reads as follows:

\begin{lemma}\label{ligneparligne}
Let $\calS$ be a linear subspace of $\Mat_{n,p}(\K)$, and $F$ be a range-compatible homomorphism (respectively, linear map) on $\calS$.
Fix $i \in \lcro 1,n\rcro$.
For $M \in \calS$, denote by $L_i(M)$ the $i$-th row of $M$. For $X \in \K^n$, denote by $X_i$ its $i$-th entry.
Then, there is a group homomorphism (respectively, a linear form) $F_i : L_i(\calS) \rightarrow \K$ such that
$$\forall M \in \calS, \; F(M)_i=F_i(L_i(M)).$$
\end{lemma}

\subsection{A proof of the First Classification Theorem}\label{classtheo1proof}

To give a flavor of the power of the above tools, we immediately use them to
prove Theorem \ref{classtheo1}.
This proof works by induction on $\dim V$ (compare this with the one in \cite{dSPclass}, which
uses an induction on $\dim U$ instead).
For $\dim V=1$, the result is void. Assume now that $\dim V \geq 2$.
Let $z \in V \setminus \{0\}$.
Using the rank theorem, we see that
\begin{equation}\label{equationcodimension}
\codim(\calS \modu z)=\codim \calS+\dim \bigl\{s \in \calS : \; \im s \subset \K z\bigr\}-\dim U.
\end{equation}
As $\codim \calS\leq \dim V-2$, there are two options:
\begin{itemize}
\item Either $\codim (\calS \modu z) \leq \dim V -3$;
\item Or $\dim \bigl\{s \in \calS : \; \im s \subset \K z\bigr\}=\dim U$,
which yields that $\calS$ contains all the linear maps from $U$ to $\K z$.
\end{itemize}
Let us now split the discussion into two cases:

\noindent \textbf{Case 1.} \\
Assume that one cannot find two linearly independent vectors $y_1$ and $y_2$ in $V$
such that $\codim (\calS \modu y_1) \leq \dim V-3$ and $\codim (\calS \modu y_2) \leq \dim V-3$.
Then, we can find a basis $(z_1,\dots,z_n)$ of $V$ in which all the vectors satisfy
$\codim (\calS \modu z_i)\geq \dim V-2$. Then, for all $i \in \lcro 1,n\rcro$,
the space $\calS$ contains all the linear operators from $U$ to $\K z_i$, whence summing
them shows that $\calS=\calL(U,V)$. We deduce from Proposition \ref{espacetotal} that every range-compatible
homomorphism on $\calS$ is local.

\noindent \textbf{Case 2.} \\
Assume that there are two linearly independent vectors $y_1$ and $y_2$ in $V$
such that $\codim (\calS \modu y_1) \leq \dim V-3$ and $\codim (\calS \modu y_2) \leq \dim V-3$.
Then, by induction we obtain that $F \modu y_1$ and $F \modu y_2$ are local, yielding
vectors $x_1$ and $x_2$ in $U$ such that $F(s)=s(x_1) \mod \K y_1$ and $F(s)=s(x_2) \mod \K y_2$
for all $s \in \calS$. If $x_1=x_2$, then
$F(s)-s(x_1) \in \K y_1 \cap \K y_2=\{0\}$ for all $s \in \calS$, whence $F$ is the evaluation at $x_1$
and we are done.

Assume finally that $x_1 \neq x_2$.
Then, we find $s(x_1-x_2)=F(s)-F(s) \mod \Vect(y_1,y_2)$ for all $s \in \calS$.
Thus, the non-zero vector $x:=x_1-x_2$ satisfies $\calS x \subset \Vect(y_1,y_2)$.
However, the space of all linear maps $s : U \rightarrow V$ satisfying $s(x) \in \Vect(y_1,y_2)$
has obviously codimension $\dim V-2$ in $\calL(U,V)$, whence it \emph{equals} $\calS$.
Thus, in well-chosen bases of $U$ and $V$, we see that $\calS$ is represented by the matrix space
$\calD \coprod \Mat_{n,p-1}(\K)$, where $\calD=\K^2 \times \{0\} \subset \K^n$ (with $n:=\dim V$ and $p:=\dim U$).
By Lemma \ref{dimU=1}, every range-compatible homomorphism on $\calD$ is local, and the same holds for
$\Mat_{n,p-1}(\K)$ by Proposition \ref{espacetotal}. Using Lemma \ref{splittinglemma}, one concludes that every range-compatible homomorphism on $\calS$
is local.

Thus, Theorem \ref{classtheo1} is proved by induction on $\dim V$.

We shall now adapt the above arguments so as to obtain the following
slight generalization of the First Classification Theorem (which will be useful in the last section of the article).

\begin{theo}[Generalized first classification theorem]\label{generalizedfirsttheo}
Let $U$ and $V$ be finite-dimensional vector spaces.
Let $\calT$ be a \emph{subgroup} of $\calL(U,V)$ which contains a linear subspace
with codimension at most $\dim V-2$, and let $F : \calT \rightarrow V$ be a range-compatible homomorphism.
Then, $F$ is local.
\end{theo}

To see how to adapt the above proof, we note first that
the line of reasoning from the proof of Lemma \ref{carachom} actually shows that, given a
subgroup $H$ of a vector space $V$ such that $\Vect(H)$ has dimension greater than $1$,
every homomorphism $u : H \rightarrow V$ which satisfies $\forall x \in H, \; u(x) \in \K x$
is a scalar multiple of the identity. Moreover, statements (a), (c) and (d) of the Splitting Lemma
can be easily extended to subgroups of matrices instead of linear subspaces.
Similarly, the projection technique still works in the context of subgroups of linear operators.
Finally, to make the inductive step work, one considers a linear subspace $\calS$ which is included in
$\calT$ and such that $\codim \calS \leq \dim V-2$, and one discusses whether
there are two linearly independent vectors $y_1$ and $y_2$ of $V$ such that
$\codim (\calS \modu y_1) \leq \dim V-3$ and $\codim (\calS \modu y_2) \leq \dim V-3$.
If there are no two such vectors, one finds $\calS=\calL(U,V)$ whence $\calT=\calS=\calL(U,V)$ and we are done. \\
If two such vectors exist and $F$ is non-local, one finds that, in well-chosen bases of $U$ and $V$,
$\calS$ is represented by $\calD \coprod \Mat_{n,p-1}(\K)$, where $n:=\dim V$, $p:=\dim U$, and
$\calD$ is a $2$-dimensional subspace of $\K^n$; then, in those bases, $\calT$ is represented by $H \coprod \Mat_{n,p-1}(\K)$
for some subgroup $H$ of $\K^n$ that contains $\calD$, and the above generalization of Lemma \ref{carachom}
shows that every range-compatible homomorphism on $H$ is local; as the same holds for $\Mat_{n,p-1}(\K)$,
one concludes that $F$ is local.

\subsection{Application to spaces of Type 1}

Now that the First Classification Theorem has been proved, we can move forward and use that theorem
to examine the range-compatible homomorphisms on spaces of Type 1:

\begin{prop}\label{vecteurdim1}
Let $U$ and $V$ be finite-dimensional vector spaces.
Let $\calS$ be a linear subspace of $\calL(U,V)$ with $\codim \calS \leq 2 \dim V-3$.
Assume that there is a non-zero vector $x$ in $U$ such that $\dim \calS x \leq 1$. \\
Let $F$ be a range-compatible homomorphism on $\calS$.
\begin{enumerate}[(a)]
\item If $\calS x=\{0\}$, then $F$ is local.
\item If $\calS x \neq \{0\}$, then $F$ is the sum of a local map and of
$s \mapsto \alpha(s(x))$ for some endomorphism $\alpha$ of $(\calS x,+)$.
\item Whenever $F$ is linear, it is local.
\end{enumerate}
\end{prop}

\begin{proof}
We lose no generality in assuming that $U=\K^p$, $V=\K^n$, $x$ is the first vector of the standard
basis of $\K^p$, and $\calS x \subset \K \times \{0\}$.
We see $\calS$ as a linear subspace of $\Mat_{n,p}(\K)$, and we split every
$M \in \calS$ as $$M=\begin{bmatrix}
C(M) & K(M)
\end{bmatrix} \quad \text{with $C(M) \in \K^n$ and $K(M) \in \Mat_{n,p-1}(\K)$.}$$
Denote by $e_1$ the first vector of the standard basis of $\K^n$.
Then, in Case (b), we have to show that $F$ is the sum of a local map and of $M \mapsto \alpha(m_{1,1})\,e_1$
for some endomorphism $\alpha$ of the group $(\K,+)$.

\noindent \textbf{Case 1.} $C(\calS)=\{0\}$. \\
Then, $\calS=\{0\} \coprod K(\calS)$.
Note that $\codim K(\calS)=\codim \calS-n \leq n-3$,
whence Theorem \ref{classtheo1} yields that every range-compatible homomorphism on $K(\calS)$ is local.
It follows from the Splitting Lemma that $F$ is local.

\noindent \textbf{Case 2.} $C(\calS)\neq \{0\}$. \\
\noindent \textbf{Subcase 2.1.} $\dim K(\calS)<\dim \calS$. \\
Then, $\dim K(\calS)=\dim \calS-1$ and
$\calS=(\K e_1) \coprod K(\calS)$. We split $F=f \coprod g$, where $f$ and $g$ are range-compatible homomorphisms, respectively, on
$\K e_1$ and on $K(\calS)$. Again, $\codim K(\calS) \leq n-2$, whence $g$ is local.
On the other hand, it is obvious that $f : M \mapsto \alpha(m_{1,1})\,e_1$ for some endomorphism $\alpha$ of $(\K,+)$.
Writing $F=(f \coprod 0)+(0 \coprod g)$ and noting that $0 \coprod g$ is local, we conclude that $F$ has the claimed form.

\noindent \textbf{Subcase 2.2.} $\dim K(\calS)=\dim \calS$. \\
For every $M \in \calS$, we further split
$$K(M)=\begin{bmatrix}
[?]_{1 \times (p-1)} \\
J(M)
\end{bmatrix},$$
so that $J(\calS)$ is a linear subspace of $\Mat_{n-1,p-1}(\K)$ with
$\codim J(\calS) \leq \codim K(\calS) \leq n-3=(n-1)-2$.
The First Classification Theorem yields that every range-compatible homomorphism on $J(\calS)$ is local.
Noting that $\calS \modu e_1$ is represented by the matrix space $\{0\} \coprod J(\calS)$,
we deduce from the Splitting Lemma that $F \modu e_1$ is local.
Thus, by subtracting a well-chosen local map from $F$, we can assume that $F \modu e_1=0$.
In other words, $F(\calS) \subset \K e_1$.
To conclude, it suffices to prove that $F$ is a group homomorphism of the entry of the matrices of $\calS$
at the $(1,1)$-spot, which amounts to proving that $F$ vanishes at every matrix of $\calH:=\Ker C$.
By a \emph{reductio ad absurdum}, we assume that this is not the case.
As $\calH=\{0\} \coprod K(\calH)$ and
$\codim K(\calH)=\codim \calH -n=\codim \calS-n+1 \leq n-2$, another application of Theorem \ref{classtheo1} yields that
the restriction of $F$ to $\calH$ is local. Thus, we have a vector $x \in \K^{n-1}$, necessarily non-zero,
such that $K(N)x \in \K e_1$ for all $N \in \calH$, whence $\codim K(\calH) \geq n-1$, contradicting what we have
just seen.
Therefore, $F$ vanishes everywhere on $\calH$, yielding a group endomorphism $\alpha$ of $(\K,+)$ such that
$F(M)=\alpha(m_{1,1})\,e_1$ for all $M \in \calS$. This completes the proof of statement (b).

Assume finally that $F$ is linear. In Case 1, we already know that $F$ is local.
In Case 2, we have shown that $F$ is the sum of a local map and of $M \mapsto \alpha(m_{1,1})\,e_1$ for some endomorphism $\alpha$ of $(\K,+)$.
As every local map is linear, we deduce that $\alpha$ is linear, which yields a scalar $\lambda$ such that $\alpha : x \mapsto \lambda x$;
thus, $M \mapsto \alpha(m_{1,1})\,e_1=M\times (\lambda e_1)$ is local, and hence $F$, being the sum of two local maps, is local.
This proves statement (c).
\end{proof}

\section{Range-compatible homomorphisms on full spaces of symmetric matrices}\label{symmetricsection}

In this section, we give a complete classification of the range-compatible homomorphisms on
the space of all $n \times n$ symmetric matrices over $\K$, i.e.\ we prove Theorem \ref{symmetrictheo}.
Let us recall its statement, with a small addition.

\begin{theo}\label{symtheorepeated}
Assume that $n \geq 2$.
If $\K$ has characteristic not $2$, then every
range-compatible homomorphism on $\Mats_n(\K)$ is local. \\
If $\K$ has characteristic $2$, then the group of all range-compatible homomorphisms on $\Mats_n(\K)$
is generated by the local maps together with the maps of the form
$$M \longmapsto \begin{bmatrix}
\alpha(m_{1,1}) & \alpha(m_{2,2}) & \cdots & \alpha(m_{n,n})\end{bmatrix}^T,$$
where $\alpha$ is a root-linear form on $\K$.

If $\K$ has more than $2$ elements, then every linear range-compatible map on $\Mats_n(\K)$ is local.
\end{theo}

Only the case $n \leq 3$ will be needed in the proof of Theorem \ref{classtheo2} but the generalization to larger values is virtually costless.
We split the proof into two main parts. It is shown in Section \ref{specialRCrootlinear} that, if $\K$ has characteristic $2$,
then the homomorphism $M \in \Mats_n(\K) \mapsto \begin{bmatrix}
\alpha(m_{1,1}) & \alpha(m_{2,2}) & \cdots & \alpha(m_{n,n})\end{bmatrix}^T$
is range-compatible for all root-linear forms $\alpha$ on $\K$. In the second part of the proof,
we shall analyze the range-compatible homomorphisms on $\Mats_n(\K)$:
this will be done by analyzing first the case $n=2$ (Section \ref{symmetricn=2}), and then by extending the result to larger values
(Section \ref{symmetricn>2}). The last statement is quickly proved in Section \ref{linearsym} with the help of a basic lemma on root-linear forms.

\subsection{Special range-compatible homomorphisms on $\Mats_n(\K)$}\label{specialRCrootlinear}

Let $n \geq 2$ be an integer, and assume that $\K$ has characteristic $2$.
Consider an arbitrary root-linear form $\alpha$ on $\K$.
We have to prove that the homomorphism
$$F : M \in \Mats_n(\K) \mapsto \begin{bmatrix}
\alpha(m_{1,1}) & \alpha(m_{2,2}) & \cdots & \alpha(m_{n,n})
\end{bmatrix}^T$$
is range-compatible.
There are two steps. First of all, we show that $F$ is range-compatible on matrices of rank at most $1$.
Then, we extend the property to every symmetric matrix.

Let $M \in \Mats_n(\K)$ be of rank $1$. Then, $M=aX X^T$ for some
$a \in \K^*$ and some $X=\begin{bmatrix}
x_1 & \cdots & x_n
\end{bmatrix}^T \in \K^n$. In particular, $m_{i,i}=a x_i^2$ for all $i \in \lcro 1,n\rcro$, whence
$$F(M)=\alpha(a) X \in \im M.$$
Thus, $F(M) \in \im M$ whenever $M \in \Mats_n(\K)$ has rank at most $1$.

Now, let $M \in \Mats_n(\K)$ be an arbitrary symmetric matrix.
If every diagonal entry of $M$ is zero, then $F(M)=0 \in \im M$.
Assume now that some diagonal entry of $M$ is non-zero. Then, by Theorem 3.0.13 of \cite[Chapter XXXV]{invitquad},
there is a diagonal matrix $D \in \Mat_n(\K)$ and a non-singular matrix $P \in \GL_n(\K)$ such that
$M=PDP^T$. Writing $D=\Diag(d_1,\dots,d_n)$, we split $D:=D_1+\cdots+D_n$ where, for $i \in \lcro 1,n\rcro$,
we have set $D_i=d_i\,E_{i,i}$.

Setting $M_i:=P D_i P^T$, we deduce that $M=M_1+\cdots+M_n$, that $\im (M)=\im(M_1)+\cdots+\im(M_n)$ and that
each matrix $M_i$ has rank at most $1$. Thus,
$$F(M)=\sum_{i=1}^n F(M_i) \in \sum_{i=1}^n \im (M_i)=\im M.$$
This completes the proof.

\begin{Rem}
Assume that $\K$ is perfect, i.e. the Frobenius endomorphism
$x \mapsto x^2$ is surjective. Then, we can give a shorter proof. In that case indeed,
$\alpha$ is a scalar multiple of the inverse $x \mapsto \sqrt{x}$ of the Frobenius automorphism of $\K$.
Then, it suffices to prove that the map
$$\varphi : M \in \Mats_n(\K) \mapsto
\begin{bmatrix}
\sqrt{m_{1,1}} & \cdots & \sqrt{m_{n,n}}
\end{bmatrix}^T$$ is range-compatible.
Let $M \in \Mats_n(\K)$. For every $X \in \Ker M$, we have
$X^TMX=0$, which reads $\underset{i=1}{\overset{n}{\sum}} x_i^2 m_{i,i}=0$, and applying $\sqrt{-}$ yields
$$\sum_{i=1}^n x_i \sqrt{m_{i,i}}=0.$$
In other words, the vector $\varphi(M)$ belongs to the orthogonal subspace of $\Ker M$, which equals $\im M$ since $M$ is symmetric.
Thus, $\varphi$ is range-compatible, as claimed.
\end{Rem}

\subsection{Analyzing the range-compatible homomorphisms on $\Mats_2(\K)$}\label{symmetricn=2}

Let $F : \Mats_2(\K) \rightarrow \K^2$ be a range-compatible homomorphism.
Using the projection technique on both rows, we obtain endomorphisms
$f,g,u,v$ of the group $\K$ such that
$$F : \begin{bmatrix}
a & b \\
b & c
\end{bmatrix} \longmapsto \begin{bmatrix}
f(a)+u(b) \\
g(c)+v(b)
\end{bmatrix}.$$
Note that analyzing the range-compatibility property only requires that we look at the images of rank $1$ matrices.
Let $x \in \K^*$ and $t \in \K$.
Then, $M:=x\begin{bmatrix}
1 & t \\
t & t^2
\end{bmatrix}$ has rank $1$ and its image is the span of $\begin{bmatrix}
1 \\
t
\end{bmatrix}$. This yields the identity
$$g(t^2x)+v(t x)=t\, f(x)+t\, u(t x).$$
Note that this identity is obvious for $x=0$, whence
\begin{equation}\label{basicsymidentity}
\forall (x,t) \in \K^2, \; g(t^2x)+v(t x)=t f(x)+t u(t x).
\end{equation}

\begin{lemma}\label{2by2identity}
Let $f,g,u,v$ be endomorphisms of $(\K,+)$ that satisfy \eqref{basicsymidentity}.
Then:
\begin{enumerate}[(a)]
\item Either $\K$ has characteristic not $2$ and then there are scalars
$\lambda$ and $\mu$ such that $f : x \mapsto \lambda\, x$, $u : x \mapsto \mu\, x$,
$v : x \mapsto \lambda\, x$ and $g : x \mapsto \mu\, x$;

\item Or $\K$ has characteristic $2$ and then there are scalars $\lambda$ and $\mu$,
together with a root-linear form $\alpha$ on $\K$ such that
$f : x \mapsto \lambda\, x+\alpha(x)$, $u : x \mapsto \mu\, x$,
$v : x \mapsto \lambda\, x$ and $g : x \mapsto \mu\, x+\alpha(x)$.
\end{enumerate}
\end{lemma}

\begin{proof}
Assume first that $\K$ has characteristic not $2$.
Then, identity \eqref{basicsymidentity} applied to $t=\pm 1$ yields
$$g-u=f-v \quad \text{and} \quad g-u=v-f,$$
whence $g=u$ and $f=v$.
Next, letting $x \in \K$ and $t \in \K$, we have
$$g(t^2x)-tg(tx)=t f(x)-f(tx).$$
Applying this to $-t$, we find $g(t^2x)-tg(tx)=-t f(x)+f(tx)$.
Thus, as $\K$ has characteristic not $2$, we deduce that $f(tx)=tf(x)$ and $g(t^2x)=tg(tx)$.
Varying $t$ and $x$ yields that $f$ and $g$ are linear, and the claimed result ensues in that case.

\vskip 3mm
Assume now that $\K$ has characteristic $2$.
Fix $x \in \K$. Then,
$$\forall t\in \K, \quad t\,u(tx)=tf(x)+v(tx)+g(t^2 x).$$
On the right-hand side of this identity is an additive map with respect to $t$.
Therefore,
$$\forall (s,t)\in \K^2, \; s\,u(tx)+t\,u(sx)=0.$$
Taking $s=1$ and varying $t$, we deduce that $u$ is linear.
Fixing $x \in \K$ and $t \in \K^*$, applying identity \eqref{basicsymidentity} to the pair $(xt^2,t^{-1})$
and multiplying it with $t$ yields
$$tg(x)+tv(tx)=f(t^2x)+u(t x).$$
As this also holds for $t=0$, we deduce that $(g,f,v,u)$ satisfies identity \eqref{basicsymidentity}, and
hence the above proof shows that $v$ is linear.

Replacing $(f,g,u,v)$ with $(f-v,g-u,0,0)$, we see that identity \eqref{basicsymidentity} is still satisfied in this new situation because $u$ and $v$
are both linear, whence no generality is lost in assuming that $u=v=0$.
Applying \eqref{basicsymidentity} with $t=1$ then yields $f=g$. Finally, identity
\eqref{basicsymidentity} yields that $f$ is a root-linear form on $\K$. This finishes the proof.
\end{proof}

With the result of Lemma \ref{2by2identity}, we find that:
\begin{itemize}
\item If $\K$ has characteristic not $2$, then there are scalars $\lambda$ and $\mu$ such that
$$\forall M=\begin{bmatrix}
a & b \\
b & c
\end{bmatrix} \in \Mats_2(\K), \quad F(M)=\begin{bmatrix}
\lambda a+\mu b \\
\lambda b+\mu c
\end{bmatrix}=M\times \begin{bmatrix}
\lambda \\
\mu
\end{bmatrix},$$
whence $F$ is local.
\item If $\K$ has characteristic $2$, then there is a root-linear form $\alpha$ on $\K$ and a pair
$(\lambda,\mu)\in \K^2$ such that
$$\forall M=\begin{bmatrix}
a & b \\
b & c
\end{bmatrix} \in \Mats_2(\K), \quad F(M)=\begin{bmatrix}
\lambda a+\mu b+\alpha(a) \\
\lambda b+\mu c+\alpha(c)
\end{bmatrix}=M \times \begin{bmatrix}
\lambda \\
\mu
\end{bmatrix}+\begin{bmatrix}
\alpha(a) \\
\alpha(c)
\end{bmatrix}.$$
\end{itemize}
Thus, the proof of Theorem \ref{symmetrictheo} is complete in the special case when $n=2$.

\subsection{Analyzing the range-compatible homomorphisms on $\Mats_n(\K)$ for $n>2$}\label{symmetricn>2}

Now, we tackle the general case. Assume that $n > 2$, and let $F : \Mats_n(\K) \rightarrow \K^n$
be a range-compatible homomorphism.
For every $S \in \Mats_2(\K)$, the vector $F(S \oplus 0_{n-2})$ must belong to $\K^2 \times \{0\}$.
Thus, we recover a range-compatible homomorphism
$$F_{1,2} : \Mats_2(\K) \rightarrow \K^2$$
such that
$$\forall S \in \Mats_2(\K),
\; F(S \oplus 0_{n-2})=\begin{bmatrix}
F_{1,2}(S) \\
[0]_{(n-2) \times 1}
\end{bmatrix}.$$
Applying the case $n=2$ to $F_{1,2}$ yields scalars $a_{1,2}$ and $a_{2,1}$ such that
$$\forall x \in \K, \; F(x E_{1,2}+x E_{2,1})=\begin{bmatrix}
a_{1,2}\,x \\
a_{2,1}\,x \\
[0]_{(n-2) \times 1}
\end{bmatrix}.$$
More generally, for every $(i,j)\in \lcro 1,n\rcro^2$ with $i \neq j$, we find
scalars $a_{i,j}$ and $a_{j,i}$ such that
$$\forall x \in \K, \; F(x E_{i,j}+x E_{j,i})=a_{i,j}\,x\, e_i+a_{j,i}\,x\, e_j,$$
where $(e_1,\dots,e_n)$ denotes the standard basis of $\K^n$.
Next, we prove that $a_{i,j}$ depends only on $j$.
Let indeed $k \in \lcro 1,n\rcro \setminus \{i,j\}$.
Then, $\im (E_{i,j}+E_{j,i}+E_{k,j}+E_{j,k})=\Vect(e_i+e_k,e_j)$, and on the other hand
$F(E_{i,j}+E_{j,i}+E_{k,j}+E_{j,k})=a_{i,j}\, e_i+a_{j,i} \,e_j+a_{k,j}\, e_k+a_{j,k}\, e_j$.
Therefore, $a_{k,j}=a_{i,j}$.
This yields scalars $y_1,\dots,y_n$ such that for all $x \in \K$ and for all distinct indices
$i$ and $j$,
$$F(x E_{i,j}+x E_{j,i})=y_j x\, e_i+y_i x\, e_j.$$
Therefore, with $Y:=\begin{bmatrix}
y_1 & \cdots & y_n
\end{bmatrix}^T$, one has $F(M)=MY$ for every symmetric matrix $M \in \Mats_n(\K)$ with diagonal zero.
As $Y \mapsto MY$ is local, we may replace $F$ with $M \mapsto F(M)-MY$.

Therefore, no generality is lost in assuming that $F$ vanishes at every symmetric matrix with diagonal zero.
From there, we split the discussion in two cases, whether $\K$ has characteristic $2$ or not.

\begin{itemize}
\item Assume that $\K$ has characteristic not $2$. Coming back to $F_{1,2}$,
we have a vector $Z \in \K^2$ such that $F_{1,2}(S)=SZ$ for all $S \in \Mats_2(\K)$.
Applying this to $S=\begin{bmatrix}
0 & 1 \\
1 & 0
\end{bmatrix}$ yields $Z=0$. Therefore $F$ vanishes at $x\,E_{1,1}$ and $x\,E_{2,2}$, for all $x \in \K$.
More generally, this line of reasoning shows that, for all distinct $i$ and $j$ in $\lcro 1,n\rcro$ and all $x \in \K$,
the map $F$ vanishes at $x\,E_{i,i}$ and $x\,E_{j,j}$. As $F$ is additive, we conclude that $F=0$.

\item Assume that $\K$ has characteristic $2$. Then, we find a vector $Z\in \K^2$ and a root-linear form $\alpha$ on $\K$ such that, for all
$S=\begin{bmatrix}
x & y \\
y & z
\end{bmatrix} \in \Mats_2(\K)$, $F_{1,2}(S)=SZ+\begin{bmatrix}
\alpha(x) \\
\alpha(z)
\end{bmatrix}$. Applying this to $S=\begin{bmatrix}
0 & 1 \\
1 & 0
\end{bmatrix}$ yields $Z=0$. Thus, $F(x\, E_{1,1})=F(x\, E_{2,2})=\alpha(x)$ for all $x \in \K$. \\
More generally, this line of reasoning yields, for all $i \in \lcro 2,n\rcro$, a root-linear form
$\beta_i$ on $\K$ such that $\forall x \in \K, \; F(x\, E_{1,1})=F(x\, E_{i,i})=\beta_i(x)$.
Obviously, $\beta_i=\alpha$ for all $i \in \lcro 2,n\rcro$.
As $F$ is additive, one concludes that
$$\forall M \in \Mats_n(\K), \; F(M)=\begin{bmatrix}
\alpha(m_{1,1}) & \alpha(m_{2,2}) & \cdots & \alpha(m_{n,n})
\end{bmatrix}^T,$$
which completes the proof.
\end{itemize}

\subsection{Linear range-compatible maps on $\Mats_n(\K)$}\label{linearsym}

We complete the proof of Theorem \ref{symtheorepeated} by considering the case of
linear range-compatible maps on $\Mats_n(\K)$. If $\K$ has characteristic not $2$, then we have
seen that every range-compatible homomorphism on $\Mats_n(\K)$ is local, whence this is also the case of
range-compatible linear maps. Assume now that $\K$ has characteristic $2$ and more than $2$ elements.
Then, the following lemma, which will be reused later in the article, is relevant to our study:

\begin{lemma}\label{linear+rootlinearLemma}
Assume that $\K$ has characteristic $2$ and more than $2$ elements.
Let $\alpha : U \rightarrow V$ be a map that is both linear and root-linear. Then, $\alpha=0$.
\end{lemma}

Over $\F_2$, the root-linear maps are the linear maps, which explains the restriction on the cardinality of $\K$.

\begin{proof}
Let $x \in U$. Choose $\lambda \in \K \setminus \{0,1\}$. Then,
$(\lambda^2-\lambda)\alpha(x)=\alpha(\lambda x)-\alpha(\lambda x)=0$, and hence
$\alpha(x)=0$.
\end{proof}

Now, if we let $F$ be a range-compatible linear map on $\Mats_n(\K)$, then we know from
Theorem \ref{symmetrictheo} that $F$ is the sum of a local map and of a map $M \mapsto \begin{bmatrix}
\alpha(m_{1,1}) & \cdots & \alpha(m_{n,n})
\end{bmatrix}^T$ for some root-linear form $\alpha$ on $\K$.
As $F$ is linear, we see that $\alpha$ is also linear, whence it is zero. Therefore, $F$ is local.
Thus, the proof of Theorem \ref{symtheorepeated} is now complete.

\subsection{Application to spaces of Type 2 or 3}

Now that Theorem \ref{symmetrictheo} has been proved, we can use it, in conjunction with the Embedding Lemma, the Splitting Lemma
and Proposition \ref{espacetotal}, to obtain the following general result.
Statements (d) and (e) in Theorem \ref{classtheo2} are obvious special cases of it:

\begin{cor}\label{symmetriccor}
Let $n$ and $p$ be non-negative integers, and let $r \geq 2$ be an integer. Then:
\begin{enumerate}[(a)]
\item Every range-compatible linear map on $\Mats_r(\K) \vee \Mat_{n,p}(\K)$ is local.
\item If $\K$ has characteristic not $2$, then every range-compatible homomorphism on
$\Mats_r(\K) \vee \Mat_{n,p}(\K)$ is local.
\item If $\K$ has characteristic $2$, then the group of all range-compatible homomorphisms on
$\Mats_r(\K) \vee \Mat_{n,p}(\K)$
is generated by the local maps together with the maps of the form
$$M \longmapsto \begin{bmatrix}
\alpha(m_{1,1}) & \alpha(m_{2,2}) & \cdots & \alpha(m_{r,r}) & [0]_{1 \times n}
\end{bmatrix}^T,$$
where $\alpha$ is a root-linear form on $\K$.
\end{enumerate}
\end{cor}

\section{Advanced techniques}\label{toolssection}

In this section, we introduce two major new tools to be used in the proofs of
Theorem \ref{maintheolin} and of the Second Classification Theorem.
Most important among them is the notion of an $\calS$-adapted vector that is suited to that theorem: we shall develop
and examine that notion in the first paragraph.

\subsection{Adapted vectors}\label{adaptedvectorssection}

\begin{Def}
Let $U$ and $V$ be finite-dimensional vector spaces and $\calS$ be a linear subspace of $\calL(U,V)$.
A non-zero vector $y \in V$ is called \textbf{$\calS$-adapted} when
$$\codim_{\calL(U,V/\K y)} (\calS\modu y) \leq 2(\dim V-1)-3.$$
\end{Def}

As in the proof of Theorem \ref{classtheo1} from Section \ref{classtheo1proof}, the motivation for introducing this notion
is that the above condition amounts to saying that $\calS\modu y$ satisfies the conditions in the Second Classification Theorem,
thereby enabling an inductive strategy to analyze the range-compatible homomorphisms on $\calS$.
Fix $y \in V \setminus \{0\}$. With the rank theorem, we find that
$$\dim (\calS \modu y)=\dim \calS -\dim \bigl\{s \in \calS : \; \im s \subset \K y\bigr\}.$$
Moreover, we see with duality that
$$\dim \bigl\{s \in \calS : \; \im s \subset \K y\bigr\}=\dim U-\dim(\calS^\bot y).$$
Therefore,
\begin{equation}\label{codimensionformula}
\codim (\calS \modu y)=\codim(\calS)-\dim(\calS^\bot y).
\end{equation}

\begin{lemma}[Adapted Vectors Lemma]\label{adaptedvectorslemma}
Let $U$ and $V$ be finite-dimensional vector spaces and $\calS$ be a linear subspace of $\calL(U,V)$,
with $n:=\dim V \geq 3$ and $p:=\dim U \geq 2$.
Assume that $\codim \calS \leq 2n-3$.
Then:
\begin{enumerate}[(a)]
\item Either the set of all non-$\calS$-adapted vectors is included in a hyperplane of $V$;
\item Or $n=3$ and $\calS$ is represented, in well-chosen bases, by one of the following matrix spaces:
\begin{enumerate}[(i)]
\item The space $\{0\} \coprod \Mat_{3,p-1}(\K)$ of all $3 \times p$ matrices with first column zero.
\item $\calK_1 \coprod \Mat_{3,p-3}(\K)$, where
$$\calK_1 :=\Biggl\{\begin{bmatrix}
0 & a & b \\
c & 0 & d \\
e & f & 0
\end{bmatrix} \mid (a,b,c,d,e,f)\in \K^6\Biggr\};$$
\item $\calK_2 \coprod \Mat_{3,p-2}(\K)$, where
$$\calK_2 :=\Biggl\{\begin{bmatrix}
a & 0 \\
0 & b \\
0 & c
\end{bmatrix} \mid (a,b,c)\in \K^3\Biggr\};$$
\item $\calK_3 \coprod \Mat_{3,p-2}(\K)$, where
$$\calK_3 :=\Biggl\{\begin{bmatrix}
a & 0 \\
0 & b \\
c & c
\end{bmatrix} \mid (a,b,c)\in \K^3\Biggr\}.$$
\end{enumerate}
\end{enumerate}
Moreover, except in Case (i) above, there is a basis of $V$ that consists only of $\calS$-adapted vectors.
\end{lemma}

\begin{proof}
Fix $y \in V \setminus \{0\}$.
Using Formula \eqref{codimensionformula}, we see that:
\begin{itemize}
\item If $\codim \calS=2n-3$, then
$y$ is non-$\calS$-adapted if and only if $\dim \calS^\bot y \leq 1$.
\item If $\codim \calS= 2n-4$, then $y$ is non-$\calS$-adapted if and only if $\calS^\bot y =\{0\}$.
\item Otherwise, every non-zero vector of $V$ is $\calS$-adapted.
\end{itemize}
Assuming that no hyperplane of $V$ contains all the non-$\calS$-adapted vectors,
we find a basis $(y_1,\dots,y_n)$ consisting of such vectors, and we aim at proving that $\calS$
is equivalent to one of the spaces listed in Case (b).

\noindent \textbf{Case 1.} $\codim \calS<2n-3$. \\
Then, for all $i \in \lcro 1,n\rcro$, we find $\calS^\bot y_i=\{0\}$, whence
$\calS^\bot=\{0\}$. Thus, $\calS=\calL(U,V)$ and we see that $\codim \calS <2n-4$, whence
every non-zero vector of $V$ is $\calS$-adapted, contradicting our assumptions.

\noindent \textbf{Case 2.} $\codim \calS=2n-3$. \\
For $i \in \lcro 1,n\rcro$, set $D_i:=\calS^\bot y_i$. Then, we see that
$\dim D_i \leq 1$ for all $i \in \lcro 1,n\rcro$.
Denote by $\calT$ the space of all operators $t \in \calL(V,U)$ such that $t(y_i) \in D_i$ for all $i \in \lcro 1,n\rcro$.
Then, we find
$$2n-3=\dim \calS^\bot \leq \dim \calT \leq \underset{i=1}{\overset{n}{\sum}}\dim D_i \leq n.$$
As $n \geq 3$, we deduce that $n=3$ and that all the above inequalities are equalities, which yields
$\dim D_i=1$ for all $i \in \lcro 1,3\rcro$, and $\calS^\bot=\calT$.
From there, we distinguish between several cases:
\begin{itemize}
\item \textbf{Subcase 2.1.} $D_1=D_2=D_3$. \\
Then, we choose a non-zero vector $x_1$ of $D_1$, which we extend into a basis $(x_1,\dots,x_p)$ of $U$. In that basis and $(y_1,y_2,y_3)$,
one sees that $\calS^\bot$ is represented by the space of all $p \times n$ matrices with all rows zero starting from the second one,
whence $\calS$ is represented by $\{0\} \coprod \Mat_{3,p-1}(\K)$.

\item \textbf{Subcase 2.2.} $\dim(D_1+D_2+D_3)=3$. \\
Then, we choose a non-zero vector $x_i$ in each $D_i$,
so that $x_1,x_2,x_3$ are linearly independent, and we extend $(x_1,x_2,x_3)$ into a basis $(x_1,\dots,x_p)$ of $U$.
In that basis and in $(y_1,y_2,y_3)$, the operator space
$\calS^\bot$ is represented by the space of all matrices of the form $\begin{bmatrix}
a & 0 & 0 \\
0 & b & 0 \\
0 & 0 & c \\
[0]_{(p-3) \times 1} & [0]_{(p-3) \times 1} & [0]_{(p-3) \times 1}
\end{bmatrix}$ with $(a,b,c)\in \K^3$. Then, $\calS$ is represented by $\calK_1 \coprod \Mat_{3,p-3}(\K)$ in those bases.
\item \textbf{Subcase 2.3.} Exactly two of the $D_i$'s are equal. \\
Without loss of generality, we may assume that
$D_2=D_3$ and $D_2 \neq D_1$. Then, we choose a basis of $U$
in which the first two vectors belong to $D_2$ and $D_1$, respectively.
In that basis and in $(y_1,y_2,y_3)$, we find that $\calS^\bot$ is represented by the space of all
matrices of the form
$\begin{bmatrix}
0 & b & c \\
a & 0 & 0 \\
[0]_{(p-2) \times 1} & [0]_{(p-2) \times 1} & [0]_{(p-2) \times 1}
\end{bmatrix}$ with $(a,b,c)\in \K^3$, whence $\calS$ is represented by $\calK_2 \coprod \Mat_{3,p-2}(\K)$ in those bases.

\item \textbf{Subcase 2.4.} $D_1,D_2,D_3$ are pairwise distinct and
$\dim(D_1+D_2+D_3)=2$. Then we can choose non-zero vectors
$x_1 \in D_1$ and $x_2 \in D_2$ such that $x_1-x_2 \in D_3$. Extending $(x_2,x_1)$ into a basis of $U$,
we see that, in this basis and in the basis $(y_1,y_2,y_3)$ of $V$, the space $\calS^\bot$
is represented by the space of all matrices of the form
$\begin{bmatrix}
0 & b & c \\
a & 0 & -c \\
[0]_{(p-2) \times 1} & [0]_{(p-2) \times 1} & [0]_{(p-2) \times 1}
\end{bmatrix}$ with $(a,b,c)\in \K^3$.
Thus, in those bases, $\calS$ is represented by $\calK_3 \coprod \Mat_{3,p-2}(\K)$.
\end{itemize}
Finally, if we denote by $(e_1,e_2,e_3)$ the standard basis of $\K^3$, we see that
$(e_1+e_2,e_1+e_3,e_1+e_2+e_3)$ is a basis of $\K^3$ in which all the vectors are $\calS$-adapted
whenever $\calS$ equals one of the spaces $\calK_1 \coprod \Mat_{3,p-3}(\K)$, $\calK_2 \coprod \Mat_{3,p-2}(\K)$
or $\calK_3 \coprod \Mat_{3,p-2}(\K)$. Thus, there is a basis of $\calS$-adapted vectors unless $\calS$
falls into Case (i) above.
\end{proof}

\subsection{The three vectors lemma}

\begin{lemma}[Three vectors lemma]\label{3vectorslemma}
Let $\calS$ be a linear subspace of $\calL(U,V)$ with $\codim \calS \leq 2\dim V-3$.
Let $F : \calS \rightarrow V$ be a range-compatible homomorphism.
Assume that there are three linearly independent vectors $y_1$, $y_2$ and $y_3$ of $V$
such that $F\modu y_1$, $F\modu y_2$, $F\modu y_3$ are all local.
Then, $F$ is local.
\end{lemma}

\begin{proof}
We use a \emph{reductio ad absurdum}, by assuming that $F$ is non-local. Set $n:=\dim V$ and $p:=\dim U$.

Let us choose three vectors $x_1,x_2,x_3$ of $U$ such that $F \modu y_i$ is the evaluation at
$x_i$ for all $i \in \lcro 1,3\rcro$.
Note that none of our assumptions is lost in subtracting a local map from $F$, whence
we can subtract a fixed vector to each $x_i$ without losing any generality.
Assume that $x_1=x_2$. Then, by subtracting $x_1$, we may assume that $x_1=x_2=0$; in that case,
for every $s \in \calS$, we find that $F(s)$ belongs to both $\K y_1$ and $\K y_2$, whence $F(s)=0$.
Thus, $x_1 \neq x_2$ and more generally we obtain that $x_1,x_2,x_3$ are pairwise distinct.

Once more, no generality is lost in assuming that $x_1=0$, whence $x_2$ and $x_3$ are distinct non-zero vectors.
Fix $s \in \calS$. Then, $F(s) \in \K y_1$ since $x_1=0$.
On the other hand, $s(x_2)-F(s) \in \K y_2$ and $s(x_3)-F(s) \in \K y_3$, and hence
$s(x_2) \in \Vect(y_1,y_2)$, $s(x_3) \in \Vect(y_1,y_3)$ and $s(x_2-x_3) \in \Vect(y_2,y_3)$.
If $x_2$ and $x_2-x_3$ were linearly dependent, it would ensue, for all $s \in \calS$, that
$s(x_2) \in \K y_2$, which would entail that $F(s) \in \K y_2$ and hence $F(s) \in \K y_1 \cap \K y_2=\{0\}$;
that would be absurd as $F$ is non-zero.
Therefore, $x_2$ and $x_2-x_3$ are linearly independent, and we can find a basis
$\bfB$ of $U$ in which they are the first two vectors. We also extend $(y_1,y_2,y_3)$ into a basis $\bfC$ of $V$.

Finally, we note that $s(x_2)-s(x_2-x_3)=s(x_3) \in \Vect(y_1,y_3)$ for all $s \in \calS$.
Denoting by $\calM$ the space of $n \times p$ matrices representing the operators of $\calS$
in the bases $\bfB$ and $\bfC$, we deduce from the above results that for every matrix $M$ in $\calM$,
there is a triple $(a,b,c)\in \K^3$ for which the first two columns of $M$ are
$$\begin{bmatrix}
a \\
b \\
0 \\
[0]_{(n-3) \times 1}
\end{bmatrix} \quad \text{and} \quad
\begin{bmatrix}
0 \\
b \\
c \\
[0]_{(n-3) \times 1}
\end{bmatrix},$$
whereas
$$F_{\bfB,\bfC}(M)=\begin{bmatrix}
a \\
0 \\
0 \\
[0]_{(n-3) \times 1}
\end{bmatrix}.$$

However, the codimension of the space of all $n \times p$ matrices in which the first two columns have the above form
is $2n-3$, and hence $\calM$ is precisely that space. In particular, taking matrices with all columns zero starting from the third one,
we find that the vector $\begin{bmatrix}
1 \\
[0]_{(n-1) \times 1}
\end{bmatrix}$ is a linear combination of
$\begin{bmatrix}
1 \\
1 \\
0 \\
[0]_{(n-3) \times 1}
\end{bmatrix}$ and
$\begin{bmatrix}
0 \\
1 \\
1 \\
[0]_{(n-3) \times 1}
\end{bmatrix}$, which is false because
$\begin{vmatrix}
1 & 1 & 0 \\
0 & 1 & 1 \\
0 & 0 & 1
\end{vmatrix}=1$. This contradiction concludes the proof.
\end{proof}

\section{Range-compatible linear maps}

In this short section, we use the techniques featured in the preceding ones to wrap up the proof of Theorem \ref{maintheolin}.
We shall start with a brief discussion of the case $n=2$ that will be reused when we deal with non-linear range-compatible homomorphisms
in the next section. Then, we shall complete the proof of Theorem \ref{maintheolin}, first for fields with more than $2$ elements,
and then for fields with $2$ elements.

\subsection{Preliminary work for the case $n=2$}\label{n=2section}

Let $\calS$ be a linear subspace of $\Mat_{2,p}(\K)$ with codimension at most $1$. \\
If $\codim \calS=0$, then $\calS=\Mat_{2,p}(\K)$. \\
Assume now that $\codim \calS=1$. Then, $\calS^\bot$ has dimension $1$.
Choosing $A \in \calS^\bot \setminus \{0\}$, there are two cases to consider:

\noindent \textbf{Case 1.} $A$ has rank $1$. \\
Then, $A$ is equivalent to $E_{1,2}$ and hence $\calS$ is equivalent to
$\calS =\calD \coprod \Mat_{2,p-1}(\K)$, where $\calD=\K \times \{0\}$.
In particular, $\calS$ has Type~1.

\noindent \textbf{Case 2.} $A$ has rank $2$. \\
Then, $A$ is equivalent to $\begin{bmatrix}
0 & -1 \\
1 & 0 \\
[0]_{(p-2) \times 1} & [0]_{(p-2) \times 1}
\end{bmatrix}$, which shows that $\calS$ has Type~2.

\subsection{Proof of Theorem \ref{maintheolin} for fields with more than $2$ elements}\label{linearmorethan2}

Here, we assume that $\# \K>2$.
The proof works by induction on $n$.
The conditions cannot be met if $n=1$. If $n=2$,
then we have seen in the above section that $\calS$ is represented by one of the matrix spaces
$\Mat_{2,p}(\K)$, $(\K \times \{0\}) \coprod \Mat_{2,p-1}(\K)$ or $\Mats_2(\K) \coprod \Mat_{2,p-2}(\K)$, where $p=\dim U$.
In each case, we see that every range-compatible linear map on $\calS$ is local. Indeed, in the first case
this follows directly from Proposition \ref{espacetotal}; in the second one, this follows from Proposition \ref{espacetotal}
combined with the Splitting Lemma and Lemma \ref{dimU=1}; in the last case, this follows from Corollary \ref{symmetriccor}.

Now, assume that $n \geq 3$. Let $F : \calS \rightarrow V$ be a linear range-compatible map.
If $\calS$ is represented by $\{0\} \coprod \Mat_{3,p-1}(\K)$, then we see that $F$ is local
by using Proposition \ref{espacetotal} together with the Splitting Lemma.
Let us now assume that the contrary holds. Then, Lemma \ref{adaptedvectorslemma} yields a basis $(y_1,\dots,y_n)$ of $V$ in which all the vectors are $\calS$-adapted. For all $i \in \lcro 1,n\rcro$, we have $\dim (\calS \modu y_i) \leq 2(n-1)-3$, whence by induction
$F\modu y_i$ is local. As $n \geq 3$, Lemma \ref{3vectorslemma} yields that $F$ is local,
which concludes the proof.

\subsection{Proof of Theorem \ref{maintheolin} for $\F_2$}\label{linear2}

In this section, we complete the proof of Theorem \ref{maintheolin} by considering the case when $\K=\F_2$.
The strategy is largely similar to the above one, with additional technicalities.
The proof involves two basic lemmas:

\begin{lemma}\label{homogeneouspolynomiallemma}
Let $\calT$ be a vector space of linear operators between finite-dimensional vector spaces $V_1$ and $V_2$.
Let $r$ be a non-negative integer, and $t \in \calT$ be such that $\rk t \geq r$. Then,
there exists an $r$-homogeneous polynomial function $h$ on $\calT$ such that $h(t)\neq 0$
and $h$ vanishes at every operator of $\calT$ with rank less than $r$.
\end{lemma}

\begin{proof}
Set $s:=\rk(t)$, $n:=\dim V_2$ and $p:=\dim V_1$.
Choosing bases of $V_1$ and $V_2$ according to $t$, we reduce the situation to the one where
$\calT$ is a linear subspace of $\Mat_{n,p}(\K)$ and
$$t=\begin{bmatrix}
I_s & [0]_{s \times (p-s)} \\
[0]_{(n-s) \times s} & [0]_{(n-s) \times (p-s)}
\end{bmatrix}.$$
Then, for $M \in \calT$, we define $h(M)$ as the determinant of the $r \times r$ upper-left submatrix of $M$.
It is then easily checked that $h : \calT \rightarrow \K$ has the claimed properties.
\end{proof}

\begin{lemma}\label{quadformlemma}
Let $q$ be a non-zero quadratic form on a finite-dimensional vector space $V$ over $\F_2$. \\
Set $n:=\dim V$.
Then, $q^{-1}\{1\}$ is not included in an $(n-2)$-dimensional linear subspace of $V$.
\end{lemma}

\begin{proof}
Assume that there is a linear subspace $G$ of codimension $2$ in $V$ which contains $q^{-1}\{1\}$.
Denote by $b$ the polar form of $q$.
Let $e_1 \in V \setminus G$.
Then, $H:=\K e_1+G$ is a linear hyperplane of $V$, to the effect that $V \setminus H$
spans $V$. Thus, we can extend $e_1$ into a basis $(e_2,\dots,e_n)$ of $V$ in which $e_i \not\in H$ for all $i \in \lcro 2,n\rcro$.
From there, one sees that $\Vect(e_1,e_i) \cap G=\{0\}$ for all $i \in \lcro 2,n\rcro$.
Fixing $i \in \lcro 1,n\rcro$, we deduce that $q$ vanishes everywhere on $\Vect(e_1,e_i)$,
and in particular $b(e_1,e_i)=0$. It follows that $e_1$ belongs to the radical of $b$.
Varying $e_1$, we deduce that $b=0$. Thus, $q$ is a linear form on $V$. As $q$ is non-zero,
$q^{-1}\{1\}$ is a non-linear hyperplane of $V$, which contradicts the assumption that $q^{-1}\{1\} \subset G$.
\end{proof}

Now, we prove Theorem \ref{maintheolin} by induction on $\dim V$.
If $\dim V=2$ the result is already known (see Proposition \ref{espacetotal} or, alternatively,
Theorem \ref{classtheo1}).
Assume now that $\dim V \geq 3$.
Let us say that a non-zero vector $z$ of $V$ is \textbf{super-$\calS$-adapted} when
$\codim(\calS \modu z) \leq 2(\dim V-1)-4$. As $\codim \calS \leq 2 \dim V-4$,
one sees from identity \eqref{codimensionformula} that
a non-zero vector $z$ is super-$\calS$-adapted whenever $\dim \calS^\bot z \geq 2$,
which reads $\rk \widehat{z} \geq 2$ if we denote by $\widehat{z}$ the operator $t \in \calS^\bot \mapsto t(z)$.

Now, the discussion splits into two main cases:

\noindent \textbf{Case 1.}
One has $\rk \widehat{z} \leq 1$ for all $z \in V$.
By the classification of vector spaces of operators with rank at most $1$, there are two options:
\begin{itemize}
\item \textbf{Subcase 1.1.} Some $1$-dimensional subspace $D$ of $U$ contains the range of every operator $\widehat{z}$. \\
Then, $\im t \subset D$ for all $t \in \calS^\bot$.
Choosing a non-zero vector of $D$ and extending it into a basis of $U$, we find that, in that basis and an arbitrary basis of $V$,
the space $\calS$ is represented by $\calD \coprod \Mat_{n,p-1}(\K)$ for some linear subspace
$\calD$ of $\K^n$. As every range-compatible linear map on $\calD$ is local, we deduce from Proposition \ref{espacetotal} and
the Splitting Lemma that every range-compatible linear map on $\calS$ is local.

\item \textbf{Subcase 1.2.} There is a hyperplane $H$ of $\calS^\bot$ on the whole of which all the operators $\widehat{z}$ vanish. \\
Obviously, no non-zero operator in $\calS^\bot$ is annihilated by all the operators $\widehat{z}$, whence $\dim \calS^\bot=1$.
Then, $\codim \calS=1$ and the conclusion follows directly from Theorem \ref{classtheo1}.
\end{itemize}

\noindent \textbf{Case 2.} Some vector $z_0 \in V$ satisfies $\rk \widehat{z_0} \geq 2$. \\
By Lemma \ref{homogeneouspolynomiallemma}, we can find a quadratic form $q$ on $V$ such that
$q(z_0) \neq 0$ and $\rk \widehat{z} \geq 2$ whenever $q(z) \neq 0$.
From there, we have two additional subcases to discuss.

\noindent \textbf{Subcase 2.1.} $\dim V>3$. \\
Lemma \ref{quadformlemma} yields linearly independent vectors $y_1,y_2,y_3$ of $V$ such that
$\rk \widehat{y_i} \geq 2$ for all $i \in \{1,2,3\}$.
Then, for all $i \in \{1,2,3\}$, we have $\codim (\calS \modu y_i) \leq 2(\dim V-1)-4$,
and by induction we deduce that $F \modu y_i$ is local.
By Lemma \ref{3vectorslemma}, it ensues that $F$ is local.

\noindent \textbf{Subcase 2.2.} $\dim V=3$. \\
Lemma \ref{quadformlemma} yields linearly independent vectors $y_1,y_2$ of $V$ such that
$\rk \widehat{y_i} \geq 2$ for all $i \in \{1,2\}$.
In that case, we actually have $\codim (\calS \modu y_1)=\codim (\calS \modu y_2)=0$, whence
$\calS\modu y_1=\calL(U,V/\K y_1)$ and $\calS\modu y_2=\calL(U,V/\K y_2)$.
By Proposition \ref{espacetotal}, we obtain two vectors $x_1$ and $x_2$ in $U$ such that $F(s)=s(x_1) \mod \K y_1$ and
$F(s)=s(x_2) \mod \K y_2$ for all $s \in \calS$.
Subtracting $s \mapsto s(x_1)$ from $F$, we see that no generality is lost in assuming that $x_1=0$,
so that $F(s) \in \K y_1$ for all $s \in \calS$.
It follows that $s(x_2) \in \Vect(y_1,y_2)$ for all $s \in \calS$. \\
If $x_2 \neq 0$, identity $\calS\modu y_2=\calL(U,V/\K y_2)$ shows that
we can choose an operator in $\calS\modu y_2$ that assigns to $x_2$
a vector outside of $\Vect(y_1,y_2)/\K y_2$, contradicting the above result.
Thus, $x_2=0$. Therefore, we have $F(s) \in \K y_1 \cap \K y_2=\{0\}$ for all $s \in \calS$,
whence $F=0$. This completes the proof.

Thus, Theorem \ref{maintheolin} is now established for all fields.

\section{Non-linear range-compatible homomorphisms}\label{nonlinearsection}

This section is devoted to the last remaining point in Theorem \ref{classtheo2},
that is the proof of statement (b). Thus, we have to show, with the assumptions of Theorem \ref{classtheo2},
that if $\calS$ is of none of Types 1 to 3 and $\dim \calS \leq 2\dim V-3$, then every range-compatible homomorphism on $\calS$ is local.
Again, the proof is done by induction on dimension of $V$, using
adapted vectors. We set $n:=\dim V$ and $p:=\dim U$.

In the situation where $n>2$, the basic idea is to assume that
$\calS$ is neither of Type 1 nor of Type 3 and that there is a non-linear range-compatible homomorphism on $\calS$.
Then, we shall slowly progress towards the conclusion that $\calS$ has Type 2.

A problem we need to get rid of is the situation where $V$ has a basis of non-$\calS$-adapted vectors:
we have seen in Lemma \ref{adaptedvectorslemma} that this situation arises only when $n=3$ and
$\calS$ has one of four very specific types.
In Section \ref{specialn=3section}, we shall give a direct proof that the excepted conclusion holds in those situations (along with an additional one that is to be encountered later).
Then, we will be ready to start our inductive proof (Section \ref{startinduction}).

\subsection{Preliminary work for $n=3$}\label{specialn=3section}

\begin{lemma}\label{specialn=3lemma}
Let $\K$ be an arbitrary field.
In addition to the notation $\calK_1$, $\calK_2$ and $\calK_3$ from Lemma \ref{adaptedvectorslemma},
we set
$$\calK_4:=\biggl\{ \begin{bmatrix}
b & c \\
a & 0 \\
0 & a
\end{bmatrix} \mid (a,b,c)\in \K^3\biggr\}.$$
For all $i \in \{1,3,4\}$, every range-compatible homomorphism on $\calK_i$ is local. \\
\end{lemma}

\begin{proof}
We can split $\calK_1 =\calD_1 \coprod \calD_2 \coprod \calD_3$, where $\calD_1$, $\calD_2$ and $\calD_3$ are $2$-dimensional subspaces of
$\K^3$. By Lemma \ref{dimU=1}, every range-compatible homomorphism on $\calD_1$ is local, and the same holds for $\calD_2$ and $\calD_3$.
Using the Splitting Lemma, one concludes that the same holds for $\calK_1$.

Let $F : \calK_3 \rightarrow \K^3$ be a range-compatible homomorphism. Using Lemma \ref{ligneparligne},
we obtain endomorphisms
$f$, $g$ and $h$ of $(\K,+)$ such that
$$F :  \begin{bmatrix}
a & 0 \\
0 & b \\
c & c
\end{bmatrix}\longmapsto \begin{bmatrix}
f(a) \\
g(b) \\
h(c)
\end{bmatrix}.$$
It follows that, for every $(a,b,c)\in \K^3$,
$$0=\begin{vmatrix}
a & 0 & f(a) \\
0 & b & g(b) \\
c & c & h(c)
\end{vmatrix}=ab\, h(c)-bc\,f(a)-ac\,g(b).$$
Fixing $a=b=1$ and varying $c$, we find that $h$ is linear. Similarly, one finds that $f$ and $g$ are linear.
Finally, $h(1)=f(1)+g(1)$ by taking $a=b=c=1$. This yields a pair $(\lambda,\mu)\in \K^2$ such that
$f=\lambda \id_\K$, $g=\mu \id_\K$ and $h=(\lambda+\mu)\id_\K$. One then checks that $F(M)=M \times \begin{bmatrix}
\lambda\\
\mu
\end{bmatrix}$ for all $M \in \calK_3$, whence $F$ is local.

Finally, let $F : \calK_4 \rightarrow \K^3$ be a range-compatible homomorphism.
Using Lemma \ref{ligneparligne}, we find endomorphisms $f$, $g$, $h$ and $i$ of $(\K,+)$ such that
$$F : \begin{bmatrix}
b & c \\
a & 0 \\
0 & a
\end{bmatrix}\longmapsto \begin{bmatrix}
h(b)+i(c) \\
f(a) \\
g(a)
\end{bmatrix}.$$
Thus, for all $(a,b,c)\in \K^3$, we find
$$0=\begin{vmatrix}
b & c & h(b)+i(c) \\
a & 0 & f(a) \\
0 & a & g(a)
\end{vmatrix}=a^2\, h(b)+a^2\, i(c)-ac\, g(a)-ab\, f(a).$$
Thus,
$$\forall (a,b,c)\in \K^3, \; a\,h(b)+a\, i(c)-c\,g(a)-b\,f(a)=0.$$
Fixing $c=0$ and $b=1$ and varying $a$, we find that $f$ is linear. Similarly, we show that $g$, $i$ and $h$ are linear.
With $c=0$ and $b=1$, one deduces that $f=h$, and similarly one shows that $g=i$.
Thus, we have a pair $(\lambda,\mu)\in \K^2$ such that $f=\lambda \id_\K=h$ and $g=\mu \id_\K=i$,
and one checks that $F(M)=M \times \begin{bmatrix}
\lambda\\
\mu
\end{bmatrix}$ for all $M \in \calK_4$, whence $F$ is local.
\end{proof}

\subsection{Setting the induction up}\label{startinduction}

Now, we can set up our inductive proof of point (b) of Theorem \ref{classtheo2}.
Let $\calS$ be a linear subspace of $\calL(U,V)$ with $\codim \calS \leq 2n-3$ and $\# \K>2$.

In the case $n=2$, we have seen in Section \ref{n=2section} that either $\calS=\calL(U,V)$, in which case we know from Proposition \ref{espacetotal}
that every range-compatible homomorphism on $\calS$ is local, or $\calS$ has Type 1 or Type 2.

Now, we assume that $n>2$ and that
point (b) of Theorem \ref{classtheo2} holds for all target spaces $V'$ with $\dim V'=n-1$
and all linear subspaces $\calS'$ of $\calL(U,V')$ with $\codim \calS' \leq 2(n-1)-3$.
We make the following assumptions:
\begin{itemize}
\item[(A)] $\calS$ is neither of Type 1 nor of Type 3.
\item[(B)] There is a non-linear range-compatible homomorphism on $\calS$. We fix such a homomorphism and denote it by $F$.
\end{itemize}
The aim is to prove that $\calS$ is of Type 2.

Using point (a) of Proposition \ref{vecteurdim1}, we obtain the following consequence of assumptions (A)
and (B):

\begin{claim}\label{dimatleast2claim}
One has $\dim \calS x \geq 2$ for all non-zero vectors $x \in U$.
\end{claim}

Moreover, we obtain:
\begin{claim}\label{nonadaptedvectorsclaim}
The set of non-$\calS$-adapted vectors is included in a hyperplane of $V$.
\end{claim}

\begin{proof}
Assume that the contrary holds. Then, we know from Lemma \ref{adaptedvectorslemma}
that, for some non-negative integer $q$, the space $\calS$ is represented by $\{0\} \coprod \Mat_{3,q}(\K)$
or by $\calK_i \coprod \Mat_{3,q}(\K)$ for some $i \in \{1,2,3\}$.
As there is a non-local range-compatible homomorphism on $\calS$,
combining Lemma \ref{specialn=3lemma}, Proposition \ref{espacetotal} and the Splitting Lemma yields that
$\calS$ is represented, for some $q \geq 0$, by
$\calK_2 \coprod \Mat_{3,q}(\K)$ in well-chosen bases of $U$ and $V$.
This would show that $\calS$ has Type 1, contradicting assumption (A).
\end{proof}

To simplify the discourse on $\calS$-adapted vectors, we shall adopt the following definition:

\begin{Def}
Let $y \in V$ be an $\calS$-adapted vector.
We say that $y$ has Type 0 for $F$ whenever $F \modu y$ is local. \\
Given $i \in \{1,2,3\}$, we say that $y$ has Type $i$ for $F$ whenever $F \modu y$ is \emph{non-local}
and $\calS \modu y$ has Type $i$.
\end{Def}

It follows from our induction hypothesis that every $\calS$-adapted vector is of one of Types $0$ to $4$ for $F$.

\subsection{Additional key lemmas}

\begin{lemma}\label{intersectionlinearlemma}
Let $G : \calS \rightarrow V$ be a range-compatible homomorphism.
Let $V_1$ and $V_2$ be subspaces of $V$ such that $G \modu V_1$ and $G \modu V_2$ are linear.
Then, $G \modu (V_1\cap V_2)$ is linear. In particular, if $V_1 \cap V_2=\{0\}$, then $G$ is linear.
\end{lemma}

\begin{proof}
Let $(\lambda,s)\in \K \times \calS$. The assumptions show that $G(\lambda s)-\lambda G(s)$ belongs to $V_1$ and to $V_2$,
and hence it belongs to $V_1 \cap V_2$. Thus, $G(\lambda s)=\lambda G(s) \mod V_1 \cap V_2$, which yields that $G \modu (V_1\cap V_2)$ is linear.
\end{proof}

With this basic result at hand, we immediately deduce:

\begin{claim}\label{atmostonelinearclaim}
Given linearly independent vectors $y_1$ and $y_2$ of $V$, one of the maps
$F \modu y_1$ and $F \modu y_2$ is non-linear.
\end{claim}

\begin{lemma}\label{3planeslemma}
Let $\calT$ be a linear subspace of $\calL(U,V)$ with $\codim \calT \leq 2\dim V-3$. \\
Let $G : \calT \rightarrow V$ be a range-compatible homomorphism.
Assume that there are three vectors $x_1,x_2,x_3$ of $U$
such that $\calT x_1,\calT x_2,\calT x_3$ are pairwise distinct $2$-dimensional subspaces of $V$
and $G \modu \calT x_i$ is linear for all $i \in \{1,2,3\}$.
Then, $G$ is local.
\end{lemma}

\begin{proof}
Set $P_i:=\calT x_i$ for convenience.
Assume first that $x_1,x_2,x_3$ are linearly independent. From the inequalities $\dim \calT x_i \leq 2$,
we deduce that $\codim \calT \geq 3(n-2)$. As $\codim \calT \leq 2n-3$, the only option is that
$n=3$, $\dim \calT x_i=2$ for all $i$, and $\calT$ is the space of all operators
$s \in \calL(U,V)$ such that $s(x_i) \in P_i$
for all $i$. In well-chosen bases of $U$ and $V$,
we see that $\calT$ is represented by $\calD_1 \coprod \calD_2 \coprod \calD_3 \coprod \Mat_{3,q}(\K)$,
where each $\calD_i$ is a $2$-dimensional subspace of $\Mat_{3,1}(\K)$, and $q=\dim U-3$.
Applying Lemma \ref{dimU=1}, Proposition \ref{espacetotal} and the Splitting Lemma, we deduce that $G$ is local.

In the rest of the proof, we assume that $x_1,x_2,x_3$ are linearly dependent.
Obviously, they are also pairwise non-colinear. As no generality is lost in multiplying each vector with a non-zero scalar,
we can assume that $x_3=x_1+x_2$. Then, $P_3 \subset P_1+P_2$.
Next, since $G \modu P_i$ is linear for all $i \in \{1,2,3\}$,
the map  $G \modu D$ is linear, where $D:=P_1 \cap P_2 \cap P_3$.
If $D=\{0\}$, this means that $G$ is linear, and point (a) of Theorem \ref{classtheo2} yields that $G$ is local.

Assume now that $D \neq \{0\}$, to the effect that $D$ has dimension $1$.
It follows that $P_1+P_2$ has dimension $3$. Finally, since $P_1$, $P_2$ and $P_3$ are distinct hyperplanes
of $P_1+P_2$ that all contain $D$, we can find a basis $(y_1,y_2,y_3)$ of $P_1+P_2$ in which
$y_1 \in D$, $y_2 \in P_1$, $y_3 \in P_2$ and $y_2+y_3 \in P_3$.
Let us extend both $(x_1,x_2)$ and $(y_1,y_2,y_3)$ into bases $\bfB$ and $\bfC$, respectively, of $U$ and $V$.
Denoting by $\calM$ the matrix space representing $\calT$ in those bases, we see that
every matrix of $\calM$ has its left $n \times 2$ submatrix of the form
$$\begin{bmatrix}
b & c \\
a & 0 \\
0 & a \\
[0]_{(n-3) \times 1} & [0]_{(n-3) \times 1}
\end{bmatrix} \quad \text{with $(a,b,c)\in \K^3$}.$$

However, the space of all matrices with this shape has codimension $2n-3$ in $\Mat_{n,p}(\K)$, whence it equals $\calM$.
It follows that $\calM =\calR \coprod  \Mat_{n,p-2}(\K)$, where $\calR$ is the space of all
$n \times 2$ matrices of the form
$\begin{bmatrix}
N \\
[0]_{n-3) \times 2}
\end{bmatrix}$ with $N \in \calK_4$.
From Lemma \ref{specialn=3lemma}, we know that every range-compatible homomorphism on $\calK_4$ is local,
whence, by the Embedding Lemma, this holds for $\calR$ as well. Using Proposition \ref{espacetotal} and the Splitting Lemma, we conclude that $G$ is local.
\end{proof}

\subsection{Completing the proof for fields of characteristic not $2$}

Using Lemma \ref{3planeslemma}, we shall now complete the proof for fields of characteristic not $2$
as a consequence of the following more general result:

\begin{claim}\label{adaptedtype2or3claim}
There is an $\calS$-adapted vector of Type 2 or 3 for $F$.
\end{claim}

\begin{proof}
Assume on the contrary that no such vector exists. By the induction hypothesis, we deduce that for every $\calS$-adapted vector $y$,
either $F \modu y$ has Type 1 or $F \modu y$ is local.

By Claim \ref{nonadaptedvectorsclaim}, one can find non-colinear $\calS$-adapted vectors.
Using Claim \ref{atmostonelinearclaim}, we deduce that at least one $\calS$-adapted vector
$y_1$ has Type $1$ for $F$.
Then, we recover a non-zero vector $x_1 \in U$ such that $\dim(\calS \modu y_1) x_1 = 1$
and $F \modu y_1$ is the sum of a local map and of a range-compatible homomorphism whose range is included in $(\calS \modu y_1) x_1$.

From Claim \ref{dimatleast2claim}, we deduce that $\dim \calS x_1=2$, and on the other hand we obtain
that $F \modu \calS x_1$ is linear. Note also that $y_1 \in \calS x_1$.

Next, we can choose an $\calS$-adapted vector $y_2$ that does not belong to $\calS x_1$.
If $F \modu y_2$ were linear, then, as $\K y_2 \cap \calS x_1=\{0\}$ and $F \modu \calS x_1$ is linear,
we would deduce from Lemma \ref{intersectionlinearlemma} that $F$ is linear, contradicting our assumptions.
Thus, $y_2$ has Type 1 for $F$.
As above, this yields a vector $x_2$ of $U$ such that $\dim \calS x_2=2$, $F \modu \calS x_2$ is linear
and $y_2 \in \calS x_2$. Note that $\calS x_1 \neq \calS x_2$.

Finally, we contend that there exists an $\calS$-adapted vector $y_3$ outside of $\calS x_1 \cup \calS x_2$.
Indeed, we can find a hyperplane $H$ of $V$ which contains all the non-$\calS$-adapted vectors,
and then, classically (see e.g.\ Lemma 2.5 of \cite{dSPfeweigenvalues}), the proper linear
subspaces $H$, $\calS x_1$ and $\calS x_2$ cannot cover $V$ since
$\K$ has more than $2$ elements.
Using such a vector $y_3$, we proceed as above and find a vector $x_3 \in U$ with $\dim \calS x_3=2$,
$\calS x_3 \neq \calS x_1$, $\calS x_3 \neq \calS x_2$, and $F \modu \calS x_3$ linear.
Then, using Lemma \ref{3planeslemma}, we conclude that $F$ is local, which contradicts assumption (B).
\end{proof}

As a consequence, we obtain:

\begin{claim}
The field $\K$ has characteristic $2$.
\end{claim}

It follows in particular that $\# \K \geq 4$, which will be helpful in the remainder of the proof.
Note that, at this point, the proof of point (b) is complete for all fields of characteristic not $2$!

\subsection{Reduction to $\dim U=2$}

In this section, we shall show that the situation can be reduced to the one where $\dim U=2$.
As in Section \ref{linear2}, this involves the relationship between the set of $\calS$-adapted vectors and the orthogonal $\calS^\bot$ of $\calS$. Given a vector $y \in V$ such that $\dim \calS^\bot y>2$,
we see from identity \eqref{codimensionformula} that $\dim (\calS \modu y)<2(\dim V-1)-3$,
whence $y$ is $\calS$-adapted and it can be neither of Type 2 nor of Type 3 for $F$. This motivates the following definition:

\begin{Def}
A non-zero vector $y \in V$ is called \textbf{super-$\calS$-adapted} when
$\dim \calS^\bot y>2$.
\end{Def}

Our aim now is to show that the existence of such a vector would yield that every range-compatible homomorphism on $\calS$ is local.
This involves two steps:

\begin{claim}\label{structureofsuperSadaptedclaim}
Assume that some vector of $V$ is super-$\calS$-adapted. \\
If $n=3$, then there are at least two linearly independent super-$\calS$-adapted vectors. \\
If $n>3$, then the union of two $2$-dimensional linear subspaces of $V$ cannot contain all the super-$\calS$-adapted vectors.
\end{claim}

\begin{proof}
To every vector $y \in V$, we assign the linear operator $\widehat{y} : t \in \calS^\bot \mapsto t(y)$.
Then, $y$ is super-$\calS$-adapted if and only if $\rk \widehat{y}>2$. Using Lemma \ref{homogeneouspolynomiallemma},
we obtain a non-zero $3$-homogeneous polynomial function
$h : V \rightarrow \K$ that vanishes at every vector of $V$ that is not super-$\calS$-adapted.

Assume first that some hyperplane $H$ of $V$ contains all the super-$\calS$-adapted vectors.
Then, choosing a linear form $\varphi$ on $V$ with kernel $H$, we would deduce that the
$4$-homogeneous polynomial function $y \mapsto \varphi(y) h(y)$ vanishes everywhere on $V$,
which cannot hold because $\K$ has more than $3$ elements and both polynomial functions $\varphi$ and $h$ are non-zero.
It follows that $V$ contains a basis of super-$\calS$-adapted vectors, which yields the first statement.

Now, let us only assume that $n>3$, and let us consider arbitrary $2$-dimensional linear subspaces $P_1$ and $P_2$ of $V$.
If $P_1+P_2 \subsetneq V$, then we can embed $P_1+P_2$ into a hyperplane of $V$:
the above proof then shows that at least one super-$\calS$-adapted vector does not belong to that hyperplane.
Thus, it only remains to deal with the case when $\dim V=4$ and
$P_1 \oplus P_2=V$. Assume that every super-$\calS$-adapted vector belongs to $P_1 \cup P_2$.
Pick a vector $y_1$ such that $h(y_1) \neq 0$; without loss of generality, we may assume that $y_1 \in P_1$;
then, we extend $y_1$ into a basis $(y_1,y_2)$ of $P_1$, and we choose a basis $(y_3,y_4)$ of $P_2$.
One sees that the $2$-dimensional space $P:=\Vect(y_1,y_2+y_3)$ intersects $P_2$ trivially and intersects $P_1$ along $\K y_1$.
Thus, every vector of $P \setminus \K y_1$ annihilates $h$. As there are at least four $1$-dimensional linear subspaces of $P$ that are different from $\K y_1$ (remember that $\# \K>3$), this would yield $h_{|P}=0$, contradicting the fact that
$h(y_1) \neq 0$. This contradiction concludes the proof.
\end{proof}

\begin{claim}
There is no super-$\calS$-adapted vector.
\end{claim}

\begin{proof}
Assume on the contrary that there is a super-$\calS$-adapted vector.
If $n=3$, then for every such vector $y$, we see that
$\dim(\calS \modu y)=2p$, whence
$\calS \modu y=\calL(U,V/\K y)$ and $F\modu y$ must be local;
however, by Claim \ref{structureofsuperSadaptedclaim}, there are at least two non-colinear super-$\calS$-adapted vectors,
whence Claim \ref{atmostonelinearclaim} yields a contradiction.

Assume now that $n>3$. For every super-$\calS$-adapted vector $y$,
we know that either $F \modu y$ is local or $\calS \modu y$ has Type 1.
Thus, we can adapt the proof of Claim \ref{adaptedtype2or3claim} by replacing $\calS$-adapted vectors with super-$\calS$-adapted vectors.
The line of reasoning works in this context because of the second result in Claim \ref{structureofsuperSadaptedclaim}.
Then, we obtain that the assumptions of Lemma \ref{3planeslemma} are satisfied: this yields that $F$ is local,
contradicting our basic assumptions.
\end{proof}

With the same line of reasoning, one also obtains:

\begin{claim}
If $\codim \calS<2n-3$, then no vector $y \in V$ satisfies
$\dim \calS^\bot y >1$.
\end{claim}

It follows that all the operators in the dual operator space $\widehat{\calS^\bot}$ have rank at most $2$ and,
if $\codim \calS<2n-3$, they all have rank at most $1$.
In any case, it is worthwhile to note that there is no non-zero vector of $\calS^\bot$ at which all the operators of $\widehat{\calS^\bot}$
vanish, as this would yield a non-zero operator $t$ in $\calS^\bot$ for which $t(y)=0$ for all $y \in V$.
Now, we can prove:

\begin{claim}
One has $\codim \calS=2n-3$.
\end{claim}

\begin{proof}
Assume on the contrary that $\codim \calS<2n-3$.
We have shown that every non-zero operator in $\widehat{\calS^\bot}$ has rank $1$ and that
no non-zero vector of $\calS^\bot$ is annihilated by all the operators of $\widehat{\calS^\bot}$.
Using the classification theorem for vector spaces of linear operators with rank at most $1$,
we deduce that all the non-zero operators in $\widehat{\calS^\bot}$ have the same range $D$.
Thus, the range of every non-zero element of $\calS^\bot$ is $D$.
Then, in well-chosen bases of $U$ and $V$, we see that $\calS$ is represented by the matrix space
$\calD \coprod \Mat_{n,p-1}(\K)$ for some linear subspace $\calD$ of $\K^n$.
By Claim \ref{dimatleast2claim}, we find that $\dim \calD \geq 2$, whence Theorem \ref{classtheo1} shows
that every range-compatible homomorphism on $\calS$ is local, contradicting our assumptions on $F$.
\end{proof}

Now, we shall apply to $\widehat{\calS^\bot}$ the classification theorem for spaces of linear operators with rank at most $2$
(see \cite{AtkLloydPrim}):

\begin{theo}[Atkinson, Lloyd]\label{classrank2}
Let $V_1$ and $V_2$ be finite-dimensional vector spaces over a field $\K$ with more than $2$ elements.
Let $\calT$ be a linear subspace of $\calL(V_1,V_2)$ in which every operator has rank at most $2$.
Set $p:=\dim V_2$ and $n:=\dim V_1$.
Then, one of the following cases must hold:
\begin{enumerate}[(i)]
\item All the operators in $\calT$ vanish everywhere on some common $(n-2)$-dimensional subspace of $V_1$.
\item Some $2$-dimensional subspace of $V_2$ contains the range of every operator in $\calT$.
\item There is a hyperplane $H$ of $V_1$ and a $1$-dimensional subspace $D$ of $V_2$ such that every operator
$u \in \calT$ maps $H$ into $D$.
\item In some bases of $V_1$ and $V_2$, the operator space
$\calT$ is represented by the space of all matrices of the form
$$\begin{bmatrix}
A & [0]_{3 \times (n-3)} \\
[0]_{(p-3) \times 3} & [0]_{(p-3) \times (n-3)}
\end{bmatrix} \quad \text{with $A \in \Mata_3(\K).$}$$
\end{enumerate}
\end{theo}

From there, we look at each case separately.
As $n \geq 3$ and no non-zero vector of $\calS^\bot$ is annihilated by all the operators of $\widehat{\calS^\bot}$,
Case (i) is ruled out.

Assume that Case (iii) holds. As $\dim \calS^\bot=2n-3 \geq n$, we obtain an $(n-1)$-dimensional subspace
$H$ of $\calS^\bot$ and a $1$-dimensional subspace $D$ of $U$ such that
$\im t \subset D$ for all $t \in H$.
Choosing a non-zero vector $x$ of $D$, it follows that $\dim \calS x \leq 1$, contradicting Claim \ref{dimatleast2claim}.

Assume now that Case (iv) holds. Again, as no non-zero vector of $\calS^\bot$
is annihilated by all the operators of $\widehat{\calS^\bot}$, it is necessary that
$\dim \calS^\bot=3$, whence $n=3$.
Then, we can find bases of $\widehat{\calS}$ and $U$ in which $\widehat{\calS^\bot}$ is represented by the space of all matrices
of the form
$$\begin{bmatrix}
A \\
[0]_{(p-3) \times 3}
\end{bmatrix} \quad \text{with $A \in \Mata_3(\K).$}$$
Then, a straightforward computation shows that, in well-chosen bases of $V$ and $U$, the space $\calS^\bot$
is also represented by the space of all matrices of the form
$$\begin{bmatrix}
A \\
[0]_{(p-3) \times 3}
\end{bmatrix} \quad \text{with $A \in \Mata_3(\K).$}$$
Then, in the same bases, $\calS$ is represented by $\Mats_3(\K) \coprod \Mat_{3,p-3}(\K)$.
In other words, $\calS$ has Type 3, which has been ruled out from the start.

We conclude that $\widehat{\calS^\bot}$ must fall into Case (ii) from Theorem \ref{classrank2}.
This yields a $2$-dimensional subspace $U_0$ of $U$ which contains $t(y)$ for all $t \in \calS^\bot$ and $y \in V$.
Choosing a basis $\bfB$ of $U$ in which the first two vectors span $U_0$, and choosing an arbitrary basis $\bfC$ of $V$,
we obtain that the space of matrices that represents $\calS$ in those bases splits as
$\calT \coprod \Mat_{n,p-2}(\K)$ for some linear subspace $\calT$ of $\Mat_{n,2}(\K)$ with
$\dim \calT=3$. Then, we can split $F_{\bfB,\bfC}=f \coprod g$, where $f$ and $g$ are range-compatible homomorphisms,
respectively, on $\calT$ and $\Mat_{n,p-2}(\K)$. By Proposition \ref{espacetotal}, the map $g$ is local, whence
$f$ is non-local. However, if $\calT$ has Type 1, then the same holds for $\calS$; moreover, $\calT$ cannot have Type 3
since the matrices in $\calT$ have only two columns. It follows that our basic assumptions are satisfied by $\calT$.

Thus, we only need to deal with the case $\dim U=2$ as, if point (b) of Theorem \ref{classtheo2} holds in that context,
then applying it to $\calT$ yields that $\calT$ has Type 2, and from there it is obvious that $\calS$ has Type 2 itself.

Therefore, in the rest of the proof, we shall consider only the case when $p=2$.

\subsection{The case $\dim U=2$ for fields of characteristic $2$}

In this last section, we assume that $p=2$. Note that $\dim \calS=3$.
First of all, we further reduce the situation to the one where $\dim V=3$.
Indeed, we know from Claim \ref{adaptedtype2or3claim}
that there is an $\calS$-adapted vector $y_0 \in V$ with Type 2 or $3$ for $F$.
As $p=2$, the vector $y_0$ cannot have Type 3 for $F$. Then, we deduce that there are bases of $U$ and $V$ in which every
matrix representing an operator in $\calS$ can be written as
$$\begin{bmatrix}
S \\
[?]_{1 \times 2} \\
[0]_{(n-3) \times 2} \\
\end{bmatrix} \quad \text{for some $S \in \Mats_2(\K)$.}$$
Then, we have a $3$-dimensional subspace $\calT$ of $\Mat_{3,2}(\K)$
such that $\calS$ is represented, in well-chosen bases, by the space of all matrices of the form
$\begin{bmatrix}
N \\
[0]_{(n-3) \times 2}
\end{bmatrix}$ with $N \in \calT$. From there, it is obvious that $\calT$ satisfies assumptions (A) and (B);
if we prove that $\calT$ has Type 2, then it will follow that $\calS$ has Type 2.

Therefore, until the end of the proof, we need only consider the case when $\dim V=3$.
Remember that we have an $\calS$-adapted vector $y_0 \in V$ with Type 2 for $F$.
By adding to $F$ a well-chosen local map, we can assume:

\begin{itemize}
\item[(C)] $F \modu y_0$ is root-linear and non-zero.
\end{itemize}

Before we can move forward, we need a few extra results on the range-compatible homomorphisms on $\Mats_2(\K)$.

\begin{lemma}\label{rootlinnonzerolemma}
Let $G : \Mats_2(\K) \rightarrow \K^2$ be a range-compatible homomorphism.
Let $y \in \K^2 \setminus \{0\}$.
Then, $G$ is root-linear and non-zero if and only if $G \modu y$ is root-linear and non-zero. \\
Moreover $G$ is root-linear if and only if there is a root-linear form $\alpha$ on $\K$ such that
$$\forall M \in \Mats_2(\K), \quad G(M)=\begin{bmatrix}
\alpha(m_{1,1}) \\
\alpha(m_{2,2})
\end{bmatrix}.$$
\end{lemma}

\begin{proof}
The second statement follows directly from Lemma \ref{linear+rootlinearLemma}
and Theorem \ref{symmetrictheo}.
As far as the first statement is concerned, it is obvious
that $G$ is non-zero if $G \modu y$ is non-zero, and that $G \modu y$ is root-linear if $G$ is root-linear.

Now, we can find a matrix $P \in \GL_2(\K)$ with second column $y$:
since $\Mats_2(\K)=P\Mats_2(\K)P^T$ and $P\begin{bmatrix}
0 \\
1
\end{bmatrix}=y$, we see that no generality is lost in assuming that $y=\begin{bmatrix}
0 \\
1
\end{bmatrix}$. In that situation, we know that there is a vector $X \in \K^2$ and a root-linear form $\alpha$ on $\K$ such that
$$G : M \mapsto MX+\begin{bmatrix}
\alpha(m_{1,1}) \\
\alpha(m_{2,2})
\end{bmatrix}.$$
Then, $G \modu y$ is represented in well-chosen bases by
$$\varphi : L=\begin{bmatrix}
\ell_1 & \ell_2
\end{bmatrix} \in \Mat_{1,2}(\K) \mapsto LX+\alpha(\ell_1).$$
Assume that $G \modu y$ is root-linear. Then, $L \in  \Mat_{1,2}(\K) \mapsto LX$
is zero because it is both linear and root-linear. Therefore $X=0$, whence $G$ is root-linear.
If in addition we assume that $G \modu y=0$, then we further obtain $\alpha=0$, and hence $G=0$.
This concludes the proof.
\end{proof}

\begin{lemma}\label{rank1matrixlemma}
Let $G : \Mats_2(\K) \rightarrow \K^2$ be a range-compatible root-linear map.
Assume that there is a rank $1$ matrix $S \in \Mats_2(\K)$ such that $G(tS)=0$ for all $t \in \K$.
Then, $G=0$.
\end{lemma}

\begin{proof}
Using the second statement of Lemma \ref{rootlinnonzerolemma}, we find a root-linear form $\alpha$ on $\K$ such that
$$\forall M \in \Mats_2(\K), \quad G(M)=\begin{bmatrix}
\alpha(m_{1,1}) \\
\alpha(m_{2,2})
\end{bmatrix}.$$
Thus, $\alpha(t s_{1,1})=\alpha(t s_{2,2})=0$ for all $t \in \K$. \\
As $S$ has rank $1$ and is symmetric, we have $s_{1,1} \neq 0$ or $s_{2,2} \neq 0$, whence $\alpha=0$.
It ensues that $G=0$.
\end{proof}

Now, we can come back to our situation.

\begin{claim}\label{claim4}
There is no vector $y \in V \setminus \{0\}$ for which $F \modu y$ is linear.
\end{claim}

\begin{proof}
Assume that such a vector exists. Then, we can embed $y_0$ and $y$ into a $2$-dimensional subspace $P$ of $V$.
Our assumption yields that $F \modu P$ is linear. On the other hand, as $F \modu y_0$
is root-linear, non-zero, and $\calS \modu y_0$ has Type 2, we deduce from Lemma \ref{rootlinnonzerolemma} that
$F \modu P$ is root-linear and non-zero. This contradicts Lemma \ref{linear+rootlinearLemma}.
\end{proof}

\begin{claim}\label{claim5}
Let $y \in V \setminus \{0\}$ be an $\calS$-adapted vector of Type 2 for $F$.
Then, $F \modu y$ is root-linear and non-zero.
\end{claim}

\begin{proof}
Again, we embed $y$ and $y_0$ into a $2$-dimensional subspace $P$ of $V$.
Applying Lemma \ref{rootlinnonzerolemma}, we successively find that $F \modu P$ is root-linear and non-zero,
and then that $F \modu y$ is root-linear and non-zero.
\end{proof}

\begin{claim}
$F$ is root-linear.
\end{claim}

\begin{proof}
Assume that every $\calS$-adapted vector $y$ of Type 2 for $F$
is colinear to $y_0$. Then, by Claim \ref{claim4}, every $\calS$-adapted vector $y$ that is not colinear to $y_0$
must have Type 1 for $F$.
As $\# \K \geq 4$, Lemma 2.5 of \cite{dSPfeweigenvalues} yields that, for any triple of $2$-dimensional subspaces
$P_0$, $P_1$ and $P_2$ of $V$, there is a vector of $V$ outside of
$P_0 \cup P_1 \cup P_2 \cup \K y_0$.
Then, we can repeat the strategy of the proof of Claim \ref{adaptedtype2or3claim} by taking $P_0$ as a $2$-dimensional subspace containing all the non-$\calS$-adapted vectors: then, we find that the conditions of Lemma \ref{3planeslemma} are fulfilled, and hence $F$ is linear, contradicting our assumptions.

Thus, we find an $\calS$-adapted vector $y$ of Type 2 for $F$
such that $y \not\in \K y_0$. By Claim \ref{claim5}, both maps $F \modu y$ and $F \modu y_0$ are root-linear,
and we conclude that $F$ is root-linear by using the same line of reasoning as in the proof of Lemma
\ref{intersectionlinearlemma}.
\end{proof}

Now, we need to split the discussion into two cases, whether the following condition holds or not.

\begin{itemize}
\item[(F)]
Every $\calS$-adapted vector $y$ has Type 2 for $F$.
\end{itemize}

\noindent \textbf{Case 1. Condition (F) fails.} \\
Let us choose an $\calS$-adapted vector $y$ with Type 1 for $F$.
This yields a non-zero vector $x \in U$ such that $P:=\calS x$ has dimension $2$
and $F \modu P$ is linear. As $F$ is root-linear, we deduce that $F \modu P=0$.

If $y_0 \in P$, then $P/\K y_0$ is a $1$-dimensional subspace of $V/\K y_0$,
whence Lemma \ref{rootlinnonzerolemma} yields that $F \modu P=(F \modu y_0) \modu (P/\K y_0)$ is non-linear.

Therefore, $y_0 \not\in P$. Then, we extend $x$ into a basis $\bfB=(x,x')$ of $U$,
and we see that we can choose a basis $(y_1,y_2)$ of $P$ such that
$\calS \modu y_0$ is represented by $\Mats_2(\K)$ in the bases $(x,x')$ and $(\overline{y_1},\overline{y_2})$
(using the identity $P\Mats_2(\K)P^T=\Mats_2(\K)$ for all $P \in \GL_2(\K)$).
Then, as $\calS x=P$ and $F \modu y_0$ is root-linear, we deduce that there are scalars $\lambda,\mu,\nu$ such that
$\calS$ is represented in the bases $\bfB$ and $\bfC=(y_1,y_2,y_0)$ by the matrix space
$$\Biggl\{\begin{bmatrix}
a & b \\
b & c \\
0 & \lambda a +\mu b+\nu c
\end{bmatrix} \mid (a,b,c)\in \K^3\Biggr\}$$
and there is a non-zero root-linear form $\alpha$ on $\K$ such that
$$F_{\bfB,\bfC}: \begin{bmatrix}
a & b \\
b & c \\
0 & \lambda a +\mu b+\nu c
\end{bmatrix} \longmapsto \begin{bmatrix}
\alpha(a) \\
\alpha(c) \\
0
\end{bmatrix}.$$
Fix $a \in \K$ such that $\alpha(a) \neq 0$.
Thus, for all $(b,c)\in \K^2$, we find
$$0=\begin{vmatrix}
a & b & \alpha(a) \\
b & c & \alpha(c) \\
0 & \lambda a +\mu b+\nu c & 0
\end{vmatrix}=(\lambda a +\mu b+\nu c)\bigl(\alpha(a) b-\alpha(c) a\bigr).$$
Fixing $c \in \K$, we see that the polynomial function $b \mapsto \alpha(a) b-\alpha(c) a$
is non-zero of degree $1$, and $b \mapsto \lambda a +\mu b+\nu c$ has degree at most $1$.
It follows that $b \mapsto \lambda a +\mu b+\nu c$ vanishes everywhere on $\K$, whence $\mu=0$ and
$\lambda a+\nu c=0$. Varying $c$ yields $\nu=0$. As $a \neq 0$, one concludes that $\lambda=0$.
Thus, $\calS$ has Type 2, which concludes the proof in that case
(this is actually a contradiction with assumption (F), but never mind).

\noindent \textbf{Case 2. Condition (F) holds} \\
We start with a simple result.

\begin{claim}\label{claim13}
There is a rank $2$ operator $s_0$ in $\calS$ such that $t s_0 \in \Ker F$ for all $t \in \K$.
\end{claim}

\begin{proof}
As $\calS \modu y$ has Type 2 for some $y \in V \setminus \{0\}$, we see that no generality is lost in assuming that
$\calS$ is a linear subspace of $\Mat_{3,2}(\K)$ and
there is a root-linear form $\alpha$ such that every matrix of $\calS$ has the form $M=\begin{bmatrix}
a & b \\
b & c \\
? & ?
\end{bmatrix}$ and, for any such $M$, that $F(M)=\begin{bmatrix}
\alpha(a)\\
\alpha(c) \\
?
\end{bmatrix}$. Then, we can choose a matrix $M_0 \in \calS$ of the form
$M_0=\begin{bmatrix}
0 & 1 \\
1 & 0 \\
? & ?
\end{bmatrix}$, so that $M_0$ has rank $2$. \\
Fix $t \in \K \setminus \{0\}$; then, $F(tM_0)=\begin{bmatrix}
0 \\
0 \\
h(t)
\end{bmatrix}$ for some $h(t) \in \K$, and one finds $h(t)=0$ by writing that
$\begin{vmatrix}
0 & t & 0 \\
t & 0 & 0 \\
? & ? & h(t)
\end{vmatrix}=0$. Thus, $F(tM_0)=0$ for all $t \in \K$ (the case $t=0$ being trivial).
\end{proof}

From there, we choose a non-zero operator $s_0$ given by Claim \ref{claim13}. Set $P:=\im s_0$.
Assume that some vector $y \in P \setminus \{0\}$ is $\calS$-adapted.
Then, the composite of $s_0$ with the canonical projection onto $V /\K y$ has rank $1$.
As $\calS \modu y$ has Type 2, Lemma \ref{rank1matrixlemma} yields a contradiction because we know that $F \modu y$
is root-linear and non-zero.
Therefore, no vector $y \in P \setminus \{0\}$ is $\calS$-adapted.
In particular, $y_0 \not\in P$, whence, as in Case 1, we can reduce the situation to the matrix case when
$P=\K^2 \times \{0\}$, $y_0=\begin{bmatrix}
0 \\
0 \\
1
\end{bmatrix}$, and every matrix of $\calS$ has the form
$$\begin{bmatrix}
N \\
[?]_{1 \times 2}
\end{bmatrix} \quad \text{for some $N \in \Mats_2(\K)$.}$$
As no non-zero vector of $P$ is $\calS$-adapted, we deduce that for every non-zero vector $Y \in \K^2$,
there is a non-zero matrix $N$ of $\Mats_2(\K)$ with range $\K Y$ and such that $\calS$ contains
$\begin{bmatrix}
N \\
[0]_{1 \times 2}
\end{bmatrix}$. Taking successively $Y=\begin{bmatrix}
1 \\
0
\end{bmatrix}$, $Y=\begin{bmatrix}
0 \\
1
\end{bmatrix}$ and $Y=\begin{bmatrix}
1 \\
1
\end{bmatrix}$, and considering all the linear combinations of the obtained matrices,
one concludes that $\calS$ contains every matrix of the form
$\begin{bmatrix}
N \\
[0]_{1 \times 2}
\end{bmatrix}$ with $N \in \Mats_2(\K)$. As $\dim \calS=3$, one concludes that $\calS$ has Type 2.

Thus, we have completed the case $p=2$, which was the only remaining one.
This completes our inductive proof of statement (b) from Theorem \ref{classtheo2}.

\section{Application to range preservers}

\subsection{Range-restricting homomorphisms}\label{rangerestrictingsection}

In this section, we apply Theorems \ref{maintheolin} and \ref{classtheo2}
to study range-restricting homomorphisms between operator spaces.

Let us start with a few definitions:

\begin{Def}
Let $\calS$ be a linear subspace of $\calL(U,V)$, and $U'$ and $V'$ be vector spaces.
A map $F : \calS \rightarrow \calL(U',V)$ is called \textbf{range-restricting}
when $\im F(s) \subset \im s$ for all $s \in \calS$. It is called a \textbf{range preserver}
when $\im F(s)=\im s$ for all $s \in \calS$.
Note that a range-preserving homomorphism must be injective.

A map $F : \calS \rightarrow \calL(U,V')$ is called \textbf{kernel-extending}
when $\Ker s \subset \Ker F(s)$ for all $s \in \calS$. It is called a \textbf{kernel preserver}
when $\Ker s=\Ker F(s)$ for all $s \in \calS$.

We adopt similar definitions for maps between matrix spaces.
\end{Def}

As with range-compatible maps, there are obvious examples of range-restricting linear maps:
those of the form $s \mapsto s \circ u$ where $u : U' \rightarrow U$ is linear.

We shall focus on range-restricting and range-preserving maps, as the kernel-extending and kernel-preserving
ones can be readily deduced from the former by a duality argument that is more easily expressed in terms of matrices:
given a linear subspace $\calS$ of $\Mat_{n,p}(\K)$,
a map $F : \calS \rightarrow \Mat_{m,p}(\K)$ is
kernel-extending (respectively, kernel-preserving) if and only if $F^T : M \in \calS^T \rightarrow F(M^T)^T \in \Mat_{p,m}(\K)$ is range-restricting (respectively, range-preserving). Moreover, $F$ is a group homomorphism
(respectively, a linear map) if and only if $F^T$ is a group homomorphism (respectively, a linear map).

So far, the above types of preservers have not been given much attention because, in the case of full spaces
of linear maps, the linear range preservers and kernel preservers are easy to determine.
Moreover, linear range preservers are also linear rank preservers and the latter are known for a wide variety
of \emph{special} matrix spaces (e.g. full rectangular matrix spaces, spaces of symmetric matrices, of
upper-triangular matrices and so on). However, if we consider linear subspaces of matrices with small codimension, the question becomes more relevant because very little is known on preservers problem at this level of generality
(this is a fresh new research topic, with \cite{dSPlargelinpres} as the first general study of that type of problem).

Range-restricting homomorphisms are related to range-compatible homomorphisms by \emph{localizing} at the vectors of $U'$: if we have a range-restricting homomorphism $F : \calS \rightarrow \calL(U',V)$,
then, for all $x \in U'$, the localized homomorphism $F_x : s \mapsto F(s)[x]$
is range-compatible. Conversely, setting a basis $(e_i)_{i \in I}$ of $U'$ and given, for each $i \in I$,
a range-compatible homomorphism $F^{(i)}$ on $\calS$, we can glue the $F^{(i)}$'s together so as to recover
a range-restricting homomorphism $F : \calS \rightarrow \calL(U',V)$ such that $F^{(i)}=F_{e_i}$ for all $i \in I$:
more explicitly, for each $s \in \calS$, one defines $F(s)$ as the sole linear map on $U'$ which
assigns $F^{(i)}(s)$ to $e_i$ for all $i \in I$, so that $\im F(s)=\Vect\{F^{(i)}(s)\mid i \in I\} \subset \im s$.
Moreover, if each $F_{e_i}$ is the evaluation at some $x_i \in U$,
then we see that $F : s \mapsto s \circ u$, where $u : U' \rightarrow U$ is the linear operator that
assigns $x_i$ to $e_i$ for all $i \in I$.

Thus, if every range-compatible linear map on $\calS$
is local, then every range-restricting linear map from $\calS$ is the right-composition by some linear map.
The same holds if we substitute the ``linearity" assumption for the ``group homomorphism" assumption.

Thus, Theorems \ref{maintheolin} and \ref{classtheo2} yield the following result, expressed in terms of matrix spaces:

\begin{theo}\label{rangerestricterlinear}
Let $\calS$ be a linear subspace of $\Mat_{n,p}(\K)$ with
$\codim \calS \leq d_n(\K)$. Let $q \geq 1$ be an integer.
Then, the linear range-restricting maps from $\calS$ to $\Mat_{n,q}(\K)$ are the maps of the form:
$$M \longmapsto MP, \quad \text{with $P \in \Mat_{p,q}(\K)$.}$$
If, in addition, $\calS$ has none of Types 1 to 3,
then the above maps are also the range-restricting homomorphisms from $\calS$ to $\Mat_{n,q}(\K)$.
\end{theo}

With the same method, the description of all range-restricting homomorphisms from a space of Type 1, 2 or $3$
ensues from Proposition \ref{vecteurdim1} and Corollary \ref{symmetriccor}:

\begin{prop}\label{rangerestrictertype1}
Let $\calS$ be a linear subspace of $\K \vee \Mat_{n-1,p-1}(\K)$
that is not included in $\{0\} \coprod \Mat_{n,p-1}(\K)$ and such that $\codim \calS \leq 2n-3$.
Then, the range-restricting homomorphisms from $\calS$ to $\Mat_{n,q}(\K)$ are the maps of the form
$$M \mapsto MP+\begin{bmatrix}
R(m_{1,1}) \\
[0]_{(n-1) \times q}
\end{bmatrix}$$
for some $P \in \Mat_{p,q}(\K)$ and some homomorphism $R : \K \rightarrow \Mat_{1,q}(\K)$.
\end{prop}

\begin{prop}\label{rangerestrictersym}
Assume that $\K$ has characteristic $2$.
Let $r \in \lcro 2,p\rcro$ and $q \geq 1$, and set $\calS:=\Mats_r(\K) \vee \Mat_{n-r,p-r}(\K)$.
Then, the range-restricting homomorphisms from $\calS$ to $\Mat_{n,q}(\K)$ are the maps of the form
$$M \mapsto MP+\begin{bmatrix}
R(m_{1,1}) \\
\vdots \\
R(m_{r,r}) \\
[0]_{(n-r) \times q}
\end{bmatrix}$$
for some $P \in \Mat_{p,q}(\K)$ and some root-linear map $R : \K \rightarrow \Mat_{1,q}(\K)$.
\end{prop}

\subsection{Range preservers}\label{rangepreservingsection}

Now, we determine the range-preserving homomorphisms
on an operator space with small codimension.
In terms of matrices, there are two obvious examples of range-preserving linear maps:
given a linear subspace $\calS$ of $\Mat_{n,p}(\K)$, together with an integer $q \geq p$ and a non-singular matrix
$Q \in \GL_p(\K)$, the maps
$$M \in \calS \longmapsto \begin{bmatrix}
M & [0]_{n \times (q-p)}
\end{bmatrix} \in \Mat_{n,q}(\K)$$
and
$$M \in \calS \longmapsto MQ \in \Mat_{n,p}(\K)$$
are linear range preservers.
Composing maps of these two types, we obtain the \emph{standard} range preservers
$$M \in \calS \longmapsto \begin{bmatrix}
M & [0]_{n \times (q-p)}
\end{bmatrix}\times Q \in \Mat_{n,q}(\K), \quad \text{with $Q \in \GL_q(\K)$.}$$

\label{constructrangepreserver}
Let us explain how to use non-local range-compatible homomorphisms
in order to construct non-standard range-preserving homomorphisms.
Let $r \geq q \geq 1$ be integers and let $F_1,\dots,F_r$ be range-compatible homomorphisms on
a matrix space $\calS \subset \Mat_{n,p}(\K)$.
Then, we consider the space $\calT:=\calS \vee \Mat_{m,q}(\K)$ and the homomorphism
$$G : M=\begin{bmatrix}
A(M) & [?]_{n \times q} \\
[0]_{m \times p} & [?]_{m \times q}
\end{bmatrix}\in \calT \longmapsto
\begin{bmatrix}
M & [0]_{(m+n) \times (r-q)}
\end{bmatrix}+
\begin{bmatrix}
[0]_{n \times p} & F_1(A(M)) & \cdots & F_r(A(M)) \\
[0]_{m \times p} & [0]_{m \times 1} & \cdots & [0]_{m \times 1} \\
\end{bmatrix}.$$
Let $M \in \calS$. For all $i \in \lcro 1,q\rcro$, as $F_i$ is range-compatible,
we see that the $(p+i)$-th column of $G(M)$ is the sum of the $(p+i)$-th column of $M$
with a linear combination of the first $p$ columns of $M$; moreover, for all $i \in \lcro q+1,r\rcro$, the $(p+i)$-th
column of $G(M)$ is a linear combination of the first $p$ columns of $G(M)$.
On the other hand, the matrices $M$ and $G(M)$ share the same first $p$ columns.
Therefore, $G(M)$ and $M$ have the same range.
Moreover, if all the maps $F_i$ are linear, then $G$ is linear.

If all the $F_i$ maps are local, it is easy to show that $G$ is standard.
Conversely, assume that $G$ is standard. Fixing $i \in \lcro 1,r\rcro$
and denoting by $e_{p+i}$ the $(p+i)$-th vector of the standard basis of $\K^{p+r}$,
we see that $M \mapsto G(M) e_{p+i}$ must be local, and one deduces that $F_i$ is local.

Therefore, if we can find at least one non-local range-compatible homomorphism (respectively,
linear map) on $\calS$, then for all integers $r \geq q \geq 1$ and $m \geq 0$, there is a
range-preserving homomorphism (respectively, linear map) from $\calS \vee \Mat_{m,q}(\K)$
to $\calS \vee \Mat_{m,r}(\K)$ that is non-standard.
Using the examples of non-local range-compatible linear maps given in the introduction,
we deduce that, for the space $\calU$ of all $2 \times 2$ upper-triangular matrices with equal diagonal entries,
there is, for all $p \geq 1$ and all $n \geq 0$,
a non-standard range-preserving linear map from
$\calU \vee \Mat_{n,p}(\K)$ into itself.
More precisely, an example of such a map is
$$M=\begin{bmatrix}
a & b & [?]_{1 \times p} \\
0 & a & [?]_{1 \times p} \\
[0]_{n \times 1} & [0]_{n \times 1} &  [?]_{n \times p}
\end{bmatrix} \longmapsto
M+\begin{bmatrix}
0 & 0 & b & [0]_{1 \times (p-1)} \\
0 & 0 & 0 & [0]_{1 \times (p-1)} \\
[0]_{n \times 1} & [0]_{n \times 1} & [0]_{n \times 1} & [0]_{n \times (p-1)}
\end{bmatrix}.$$
Similarly, there is, for all $p \geq 1$ and all $n \geq 0$,
a non-standard range-preserving linear map from
$\Mats_2(\F_2) \vee \Mat_{n,p}(\F_2)$ into itself: an example of such a map is
$$M=\begin{bmatrix}
a & b & [?]_{1 \times p} \\
b & c & [?]_{1 \times p} \\
[0]_{n \times 1} & [0]_{n \times 1} &  [?]_{n \times p}
\end{bmatrix} \longmapsto
M+\begin{bmatrix}
0 & 0 & a & [0]_{1 \times (p-1)} \\
0 & 0 & c & [0]_{1 \times (p-1)} \\
[0]_{n \times 1} & [0]_{n \times 1} & [0]_{n \times 1} & [0]_{n \times (p-1)}
\end{bmatrix}.$$

Now, under the assumptions of Theorem \ref{maintheolin}, we can describe all the range-preserving homomorphisms.

\begin{theo}\label{rangepreservinggeneral1}
Let $\calS$ be a linear subspace of $\Mat_{n,p}(\K)$ with codimension at most $d_n(\K)$.
Assume that no non-zero vector $x$ of $\K^p$ satisfies $\calS x=\{0\}$. \\
Then:
\begin{enumerate}[(a)]
\item The linear range-preserving maps from $\calS$ to $\Mat_{n,q}(\K)$
are the maps of the form
$$M \mapsto \begin{bmatrix}
M & [0]_{n \times (q-p)}
\end{bmatrix} \times Q \quad \text{with $Q \in \GL_q(\K)$.}$$
\item If $\calS$ has none of Types 1 to 3, then the range-preserving homomorphisms are the same.
\end{enumerate}
\end{theo}

\begin{proof}
We already know that the cited maps are linear range preservers.
Conversely, let $F : \calS \rightarrow \Mat_{n,q}(\K)$ be a range-preserving homomorphism.
Assume either that $F$ is linear, or that $\calS$ is of none of Types 1 to 3.
In any case, as $F$ is a range-restricting homomorphism, Theorem \ref{rangerestricterlinear} yields some $P \in \Mat_{p,q}(\K)$
such that $F : M \mapsto MP$. Let us prove that $P$ has rank $p$.
Assume that this is not the case, set $r:=\rk P$ and split $P=Q_1 J_r Q_2$
with $J_r:=\begin{bmatrix}
I_r & [0]_{r \times (q-r)} \\
[0]_{(p-r) \times r} & [0]_{(p-r) \times (q-r)}
\end{bmatrix}$, and $Q_1 \in \GL_p(\K)$ and $Q_2 \in \GL_q(\K)$.
Then, we deduce that
$$M \in \calS Q_1 \longmapsto M J_r \in \Mat_{n,q}(\K)$$
is range-preserving.
For all $M \in \calS Q_1$, we write $M=\begin{bmatrix}
K(M) & [?]_{n \times (p-r)}
\end{bmatrix}$ with $K(M) \in \Mat_{n,r}(\K)$ and we deduce from the above proof that
$G : M \in \calS Q_1 \mapsto K(M)$ is range-preserving (and it is obviously linear).
In particular, this map is one-to-one: we deduce that
$G^{-1} : K(\calS Q_1) \rightarrow \calS Q_1$ is a linear range-preserving map
and $\codim K(\calS Q_1) \leq \codim \calS \leq d_n(\K)$, which yields a matrix $R \in \Mat_{r,p}(\K)$ such that
$M=K(M) R$ for all $M \in \calS Q_1$. As $r<p$, we can choose a non-zero vector $x \in \Ker R$,
whence $Mx=0$ for all $M \in \calS Q_1$, and finally $N(Q_1x)=0$ for all $N \in \calS$, contradicting our assumptions.

We conclude that $\rk P=p$, whence we can find $Q \in \GL_q(\K)$ such that $P=J_p Q$.
From there, we deduce that
$$\forall M \in \calS, \quad F(M)=\begin{bmatrix}
M & [0]_{n \times (q-p)}
\end{bmatrix} \times Q.$$
\end{proof}

\begin{prop}\label{rangepreservinggeneral2}
Let $\calS$ be a linear subspace of $\Mat_{n,p-1}(\K) \coprod \{0\}$ with codimension at most $2n-3$ in
$\Mat_{n,p}(\K)$. Then, the range-preserving homomorphisms from $\calS$ to $\Mat_{n,q}(\K)$
are the maps of the form
$$\begin{bmatrix}
N & [0]_{n \times 1}
\end{bmatrix} \longmapsto \begin{bmatrix}
N & [0]_{n \times (q-p+1)}
\end{bmatrix} \times Q$$
for some $Q \in \GL_q(\K)$ (thus, such a map exists if and only if $q \geq p-1$).
\end{prop}

\begin{proof}
It is obvious that the cited maps are range-preserving homomorphisms from $\calS$ to $\Mat_{n,q}(\K)$.  \\
Conversely, let us write $M=\begin{bmatrix}
K(M) & [0]_{n \times 1}
\end{bmatrix}$ for all $M \in \calS$. Let $F : \calS \rightarrow \Mat_{n,q}(\K)$ be a range-preserving homomorphism.
Then, $N \in K(\calS) \longmapsto F(K^{-1}(N)) \in \Mat_{n,q}(\K)$ is obviously a range-preserving homomorphism.
As $\codim K(\calS) \leq n-3$, we see that $K(\calS)$ has none of Types 1 to 3 and that $\codim K(\calS) \leq d_n(\K)$,
whence Theorem \ref{rangerestricterlinear} yields $Q \in \GL_q(\K)$ such that
$$\forall N \in K(\calS), \; F(K^{-1}(N))=\begin{bmatrix}
N & [0]_{n \times (q-p+1)}
\end{bmatrix} \times Q,$$
and the claimed result ensues.
\end{proof}

Now, we tackle spaces of Types 1, 2 or 3. The last two will be dealt with simultaneously.

\begin{theo}\label{rangepreservingtype1}
Let $\calS$ be a linear subspace of $\Mat_{n,p}(\K)$ with $\codim \calS \leq 2n-3$.
Assume that $\calS$ is included in $\K \vee \Mat_{n-1,p-1}(\K)$ but not in
$\{0\} \coprod \Mat_{n,p-1}(\K)$. \\
Then, the range-preserving homomorphisms from $\calS$ to $\Mat_{n,q}(\K)$
are the maps of the form
$$M=\begin{bmatrix}
m_{1,1} & L(M) \\
[0]_{(n-1) \times 1} & K(M)
\end{bmatrix} \longmapsto
\begin{bmatrix}
R(m_{1,1}) & L(M)+R'(m_{1,1}) \\
[0]_{(n-1) \times (q-p+1)} & K(M)
\end{bmatrix} \times Q,$$
where $Q \in \GL_q(\K)$, and $R : \K \rightarrow \Mat_{1,q-p+1}(\K)$ and
$R' : \K \rightarrow \Mat_{1,p-1}(\K)$ are homomorphisms such that $R$ is injective.
Such homomorphisms exist if and only if $q \geq p$.
\end{theo}

\begin{proof}
Let $Q \in \GL_q(\K)$, and let $R : \K \rightarrow \Mat_{1,q-p+1}(\K)$ and
$R' : \K \rightarrow \Mat_{1,p-1}(\K)$ be homomorphisms
such that $R$ is injective. Set
$$F : M=\begin{bmatrix}
m_{1,1} & L(M) \\
[0]_{(n-1) \times 1} & K(M)
\end{bmatrix} \in \calS \longmapsto
\begin{bmatrix}
R(m_{1,1}) & L(M)+R'(m_{1,1}) \\
[0]_{(n-1) \times (q-p+1)} & K(M)
\end{bmatrix} \times Q.$$
Then, $F$ is obviously a group homomorphism.
Let $M \in \calS$. If $m_{1,1}=0$, then $F(M)=\begin{bmatrix}
[0]_{n \times (q-p)} & M
\end{bmatrix} \times Q$ has the same range as $M$.
If $m_{1,1} \neq 0$, then $R(m_{1,1}) \neq 0$ and hence we can use a series of elementary column operations
to turn $\begin{bmatrix}
R(m_{1,1}) & L(M)+R'(m_{1,1}) \\
[0]_{(n-1) \times (q-p+1)} & K(M)
\end{bmatrix}$ into $\begin{bmatrix}
[0]_{n \times (q-p)} & M
\end{bmatrix}$, whence $F(M)$ has the same range as $M$.

Conversely, let $F : \calS \rightarrow \Mat_{n,q}(\K)$ be a range-preserving homomorphism.
Denote by $\calT$ the space of all matrices of $\calS$ with first column zero, and write every such matrix
$N \in \calT$ as
$N=\begin{bmatrix}
[0]_{n \times 1} & H(N)
\end{bmatrix}$. Note that $\codim H(\calT) \leq n-2$.
By Proposition \ref{rangerestrictertype1}, we can find a matrix $P \in \Mat_{p,q}(\K)$ together with a
group homomorphism $J : \K \rightarrow \Mat_{1,q}(\K)$ such that
$$F : M \longmapsto MP+\begin{bmatrix}
J(m_{1,1}) \\
[0]_{(n-1) \times q}
\end{bmatrix}.$$
Let us write $P=\begin{bmatrix}
[?]_{1 \times q} \\
P_1
\end{bmatrix}$
with $P_1 \in \Mat_{p-1,q}(\K)$.
In particular, we see that
$$\forall N \in \calT, \; F(N)=H(N)P_1.$$
If $P_1$ does not have rank $p-1$, we choose a non-zero row $Y \in \Mat_{1,p-1}(\K)$
such that $YP_1=0$, and we note that the space of all matrices of $\Mat_{n,p-1}(\K)$
with row space included in $\K Y$ has dimension $n$, whence it must contain a non-zero element $M$ of $H(\calT)$,
leading to $M P_1=0$. This would contradict the injectivity of $F$.
Thus, $P_1$ has rank $p-1$, to the effect that there is a matrix $Q \in \GL_q(\K)$ such that
$$P=\begin{bmatrix}
[?]_{1 \times (q-p+1)} & [?]_{1 \times (p-1)} \\
[0]_{(p-1) \times (q-p+1)} & I_{p-1}
\end{bmatrix} \times Q.$$
From there, we obtain two group homomorphisms $R : \K \rightarrow \Mat_{1,q-p+1}(\K)$ and
$R' : \K \rightarrow \Mat_{1,p-1}(\K)$ such that
$$F : M=\begin{bmatrix}
m_{1,1} & L(M) \\
[0]_{(n-1) \times 1} & K(M)
\end{bmatrix} \longmapsto
\begin{bmatrix}
R(m_{1,1}) & L(M)+R'(m_{1,1}) \\
[0]_{(n-1) \times (q-p+1)} & K(M)
\end{bmatrix} \times Q.$$
It remains to prove that $R$ is injective.
Denote by $G$ the subgroup of $\calS$ consisting of all the matrices $M$ for which $R(m_{1,1})=0$.
Note that $G$ contains the linear subspace $\calT$.
For every $M \in G$, we write $F(M)=\begin{bmatrix}
[0]_{n \times (q-p+1)} & F'(M)\end{bmatrix} \times Q$, and $F' : G \rightarrow \Mat_{n,p-1}(\K)$ is a range-preserving homomorphism whose
image we denote by $G'$. Now, let us consider the map $f : G' \rightarrow \K^n$
which assigns to every $N \in G'$ the first column of $(F')^{-1}(N)$.
One sees that $f$ is a range-compatible homomorphism.
Moreover, $G'$ contains $H(\calT)$, which is a linear subspace of $\Mat_{n,p-1}(\K)$
with codimension at most $n-2$. We deduce from Theorem \ref{generalizedfirsttheo}
that $f : N \mapsto NY$ for some $Y \in \K^{p-1}$.
With the above notation, it follows that $K(N)Y=0$ for all $N \in \calT$.
If $Y \neq 0$, this would lead to $\codim K(\calT)\geq n-1$, contradicting
$\codim K(\calT) \leq \codim H(\calT) \leq n-2$.
It follows that $f=0$, and in particular $m_{1,1}=0$ for all $M \in \calS$ satisfying $R(m_{1,1})=0$.
This completes the proof.
\end{proof}

\begin{theo}\label{rangepreservingtype23}
Assume that $\K$ has characteristic $2$.
Let $r \in \lcro 2,\min(n,p)\rcro$, and set $\calS:=\Mats_r(\K) \vee \Mat_{n-r,p-r}(\K)$.
Then, the range-preserving homomorphisms from $\calS$ to $\Mat_{n,q}(\K)$
are the maps of the form
$$M \longmapsto
\left(
\begin{bmatrix}
M & [0]_{n \times (q-p)}
\end{bmatrix}+
\begin{bmatrix}
[0]_{1 \times r} & R(m_{1,1}) \\
\vdots & \vdots \\
[0]_{1 \times r} & R(m_{r,r}) \\
[0]_{(n-r) \times r} & [0]_{(n-r) \times (q-r)}
\end{bmatrix}
\right)
\times Q$$
where $Q \in \GL_q(\K)$, and $R : \K \rightarrow \Mat_{1,q-r}(\K)$ is a root-linear map.
\end{theo}

\begin{proof}
That the maps of the given form are range-preserving homomorphisms is a consequence of
Corollary \ref{symmetriccor} and of the general construction of range-preserving maps
given in the beginning of Section \ref{rangepreservingsection}.

Conversely, let $F : \calS \rightarrow \Mat_{n,q}(\K)$ be a range-preserving homomorphism.
By Proposition \ref{rangerestrictersym}, there is a matrix $P \in \Mat_{p,q}(\K)$ and a root-linear map $R : \K \rightarrow \Mat_{1,q}(\K)$ such that
$$F : M \mapsto MP+\begin{bmatrix}
R(m_{1,1}) \\
\vdots \\
R(m_{r,r}) \\
[0]_{(n-r) \times q}
\end{bmatrix}.$$
The first step consists in showing that $P$ has rank $p$.
Assume that this is not the case. Then, there is a non-zero row matrix $L \in \Mat_{1,p}(\K)$
such that $LP=0$.
Now, denote by $\calT$ the space of all $2 \times p$ matrices of the form
$\begin{bmatrix}
0 & b & [?]_{1 \times (p-2)} \\
b & 0 & [?]_{1 \times (p-2)}
\end{bmatrix}$. Then, for all $T \in \calT$, we see that
$$F\Biggl(\begin{bmatrix}
T \\
[0]_{(n-2) \times p}
\end{bmatrix}\Biggr)=\begin{bmatrix}
T P\\
[0]_{(n-2) \times q}
\end{bmatrix},$$
and if $L$ belongs to the row space of $T$, one finds $\rk(TP) <\rk T$,
contradicting the assumption that $F$ be range-preserving.
However, $L$ belongs to the row space of some matrix in $\calT$.
To see this, we write $L=\begin{bmatrix}
a & b & L'
\end{bmatrix}$ with $(a,b)\in \K^2$ and $L' \in \Mat_{1,p-2}(\K)$.
If $a=b=0$, we see that $L$ belongs to the row space of
$\begin{bmatrix}
0 & 0 & L' \\
0 & 0 & L'
\end{bmatrix}$. If $a\neq 0$, then $L$ belongs to the row space of
$\begin{bmatrix}
0 & a & [0]_{1 \times (p-2)} \\
a & 0 & L'
\end{bmatrix}$. If $b \neq 0$, then  $L$ belongs to the row space of
$\begin{bmatrix}
0 & b & L' \\
b & 0 & [0]_{1 \times (p-2)}
\end{bmatrix}$.
Thus, a contradiction arises in any case. \\
If follows that $P$ has rank $p$,
whence we may find $Q \in \GL_q(\K)$ such that
$P=\begin{bmatrix}
I_p & [0]_{p \times (q-p)}
\end{bmatrix} \times Q$. Thus, we can find a root-linear map $R' : \K \rightarrow \Mat_{1,q}(\K)$ such that
$$\forall M \in \calS, \quad F(M)=
\left(\begin{bmatrix}
M & [0]_{n \times (q-p)}
\end{bmatrix}+
\begin{bmatrix}
R'(m_{1,1}) \\
\vdots \\
R'(m_{r,r}) \\
[0]_{(n-r) \times q}
\end{bmatrix}\right)\times Q.$$
As $M \mapsto F(M)Q^{-1}$ is also a range-preserving homomorphism, we lose no generality in assuming that $Q=I_q$
in the rest of the proof.

Now, we have root-linear form $\alpha_1,\dots,\alpha_q$ on $\K$ such that
$$\forall M \in \calS, \quad F(M)=
\begin{bmatrix}
M & [0]_{n \times (q-p)}
\end{bmatrix}+\begin{bmatrix}
\alpha_1(m_{1,1}) & \alpha_2(m_{1,1}) & \cdots & \alpha_q(m_{1,1})  \\
\alpha_1(m_{2,2}) & \alpha_2(m_{2,2}) & \cdots & \alpha_q(m_{2,2}) \\
\vdots & \vdots & & \vdots \\
\alpha_1(m_{r,r}) & \alpha_2(m_{r,r}) & \cdots & \alpha_q(m_{r,r}) \\
[0]_{(n-r) \times 1} & [0]_{(n-r) \times 1} & \cdots & [0]_{(n-r) \times 1}
\end{bmatrix}.$$
From there, we show that $\alpha_1,\dots,\alpha_r$ are all linear combinations of
$\alpha_{p+1},\dots,\alpha_q$. Let $a \in \underset{k=p+1}{\overset{q}{\bigcap}} \Ker \alpha_k$.
Let us show that $\alpha_1(a)=\cdots=\alpha_r(a)=0$. As this is trivial if $a=0$, we assume that
$a \neq 0$. Set $\lambda:=\alpha_1(a) a^{-1}$ and
$$M_0:=\begin{bmatrix}
\lambda^2 a & [0]_{1 \times (p-1)} \\
[0]_{(n-1) \times 1} & [0]_{(n-1) \times (p-1)}
\end{bmatrix}.$$
As $\alpha_1(\lambda^2 a)+\lambda^2 a=\lambda \alpha_1(a)+\lambda^2 a=0$, while
$\alpha_{p+1},\dots,\alpha_q$ vanish at $\lambda^2 a$, we see that
$$F(M_0)=\begin{bmatrix}
0 & L_0 & L_1 & [0]_{1 \times (q-p)} \\
[0]_{(n-1) \times 1} & [0]_{(n-1) \times (r-1)} & [0]_{(n-1) \times (p-r)} & [0]_{(n-1) \times (q-p)}
\end{bmatrix}$$
for some $L_0 \in \Mat_{1,r-1}(\K)$ and some $L_1 \in \Mat_{1,p-r}(\K)$.
Setting
$$M_1:=\begin{bmatrix}
\lambda^2 a & L_0 & L_1 & [0]_{1 \times (q-p)} \\
L_0^T & [0]_{(r-1) \times (r-1)} & [0]_{(r-1) \times (p-r)} & [0]_{(r-1) \times (q-p)} \\
[0]_{(n-r) \times 1} & [0]_{(n-r) \times (r-1)} & [0]_{(n-r) \times (p-r)} & [0]_{(n-r) \times (q-p)}
\end{bmatrix},$$
we deduce that
$$F(M_1)=\begin{bmatrix}
0 & [0]_{1 \times (q-1)} \\
L_0^T & [0]_{(r-1) \times (q-1)} \\
[0]_{(n-r) \times 1} & [0]_{(n-r) \times (q-1)}
\end{bmatrix}.$$
As $F$ is range-preserving, the first row of $M_1$ must be zero, which yields $\lambda^2 a=0$, whence
$\alpha_1(a)=0$.

With the same line of reasoning, one finds that $\alpha_k(a)=0$ for all $k \in \lcro 1,r\rcro$.
Thus, $\underset{k=p+1}{\overset{q}{\bigcap}} \Ker \alpha_k \subset \Ker \alpha_i$ for all $i \in \lcro 1,r\rcro$.
As $\alpha_1,\dots,\alpha_q$ are linear forms on the vector space $\K^{/2}$,
it ensues that $\alpha_1,\dots,\alpha_r$ are all linear combinations of $\alpha_{p+1},\dots,\alpha_q$.
Therefore, for a well-chosen matrix $Q_1=\begin{bmatrix}
I_r & [0]_{r \times (p-r)} & [0]_{r \times (q-p)} \\
[0]_{(p-r) \times r} & I_{p-r} & [0]_{(p-r) \times (q-p)} \\
[?]_{(q-p) \times r} & [0]_{(q-p) \times (p-r)} & I_{q-p}
\end{bmatrix}$, we have
$$\forall M \in \calS, \quad F(M)Q_1
=\begin{bmatrix}
M & [0]_{n \times (q-p)}
\end{bmatrix}+\begin{bmatrix}
[0]_{1 \times r} & \alpha_{r+1}(m_{1,1}) & \cdots & \alpha_q(m_{1,1})  \\
[0]_{1 \times r} & \alpha_{r+1}(m_{2,2}) & \cdots & \alpha_q(m_{2,2}) \\
\vdots & \vdots & & \vdots \\
[0]_{1 \times r} & \alpha_{r+1}(m_{r,r}) & \cdots & \alpha_q(m_{r,r}) \\
[0]_{(n-r) \times r} & [0]_{(n-r) \times 1} & \cdots & [0]_{(n-r) \times 1}
\end{bmatrix}.$$
Setting $R_1 : a \in \K \mapsto \begin{bmatrix}
\alpha_{r+1}(a) & \cdots & \alpha_q(a)
\end{bmatrix} \in \Mat_{1,q-r}(\K)$, we conclude that $R_1$ is root-linear and
$$\forall M \in \calS, \quad F(M)=
\left(
\begin{bmatrix}
M & [0]_{n \times (q-p)}
\end{bmatrix}+
\begin{bmatrix}
[0]_{1 \times r} & R_1(m_{1,1}) \\
\vdots & \vdots \\
[0]_{1 \times r} & R_1(m_{r,r}) \\
[0]_{(n-r) \times r} & [0]_{(n-r) \times (q-r)}
\end{bmatrix}
\right)
\times Q_1^{-1}.$$
\end{proof}

\subsection{Semi-linear range preservers}

We finish by discussing semi-linear range preservers.

\begin{cor}
Let $\calS$ be a linear subspace of $\Mat_{n,p}(\K)$ with $\codim \calS \leq d_n(\K)$.
Then, the semi-linear range-preserving maps from $\calS$ to $\Mat_{n,p}(\K)$
are the maps of the form $M \mapsto MP$ with $P \in \GL_p(\K)$.
\end{cor}

\begin{proof}
We know that $M\in \calS \mapsto MP$ is a linear range preserver for all $P \in \GL_p(\K)$.
Conversely, let $F : \calS \rightarrow \Mat_{n,p}(\K)$ be a semi-linear range preserver.
Let us prove that $F$ is linear. Since $F$ is range-restricting, we already know from Theorem \ref{rangerestricterlinear} that
it is linear whenever $\calS$ is of none of Types 1 to 3.
\begin{itemize}
\item Assume that $\calS$ has Type 1. Then, we lose no generality in assuming that $\calS$
is a subspace of $\K \vee \Mat_{n-1,p-1}(\K)$ that is not included in $\{0\} \coprod \Mat_{n,p-1}(\K)$.
The linear subspace $\calT$ of all matrices of $\calS$ with first column zero is non-zero since
$\codim \calS \leq 2n-3$. Using Theorem \ref{rangepreservingtype1}, we see that the restriction of $F$ to $\calT$ is linear
and non-zero, whence $F$ is linear.

\item Assume that $\calS$ has Type $2$ or $3$.
In any case, we may assume that, for some $r \in \lcro 2,\min(n,p)\rcro$, we have
$\calS=\Mats_r(\K) \vee \Mat_{n-r,p-r}(\K)$.
Then, Theorem \ref{rangepreservingtype23} shows that the map $A \in \Mata_r(\K) \mapsto
F\left(\begin{bmatrix}
A & [0]_{r \times (p-r)} \\
[0]_{(n-r) \times r} & [0]_{(n-r) \times (p-r)}
\end{bmatrix}\right)$ is non-zero and linear, which yields that $F$ is linear.
\end{itemize}
Thus, $F$ is linear, whence Theorem \ref{rangepreservinggeneral1} and Proposition \ref{rangepreservinggeneral2} (applied to the special case $q=p$)
yield a matrix $P \in \GL_p(\K)$ such that $F: M \mapsto MP$.
\end{proof}

The examples of the beginning of Section \ref{rangepreservingsection} show that the above result is optimal whenever $p \geq 3$.
We conclude with a straightforward application to semi-linear kernel-preserving maps:

\begin{cor}
Let $\calS$ be a linear subspace of $\Mat_{n,p}(\K)$ with $\codim \calS \leq d_p(\K)$.
Then, the semi-linear kernel-preserving maps from $\calS$ to $\Mat_{n,p}(\K)$
are the maps of the form $M \mapsto PM$ with $P \in \GL_n(\K)$.
\end{cor}


\begin{thebibliography}{1}
\bibitem{AtkLloyd}
M. D. Atkinson, S. Lloyd, Large spaces of matrices of bounded rank,
\newblock{\em Quart. J. Math. Oxford (2)}
\newblock{\textbf{31}}
\newblock{(1980)}
\newblock{253-262.}

\bibitem{AtkLloydPrim}
M. D. Atkinson, S. Lloyd, Primitive spaces of matrices of bounded rank,
\newblock{\em J. Austr. Math. Soc. (Ser. A)}
\newblock{\textbf{30}}
\newblock{(1980)}
\newblock{473-482.}

\bibitem{Dieudonne}
J. Dieudonn\'e, Sur une g\'en\'eralisation du groupe orthogonal \`a quatre variables,
\newblock{\em Arch. Math.}
\newblock{\textbf{1}}
\newblock{(1949)}
\newblock{282-287.}

\bibitem{invitquad}
C. de Seguins Pazzis, Invitation aux formes quadratiques,
\newblock{Calvage \& Mounet,}
\newblock{Paris,}
\newblock{2011.}

\bibitem{dSPfeweigenvalues}
C. de Seguins Pazzis, \emph{Spaces of matrices with few eigenvalues},
\newblock{preprint,}
\newblock{2013,}
\newblock{arXiv: http://arxiv.org/abs/1302.0301}

\bibitem{dSPclass}
C. de Seguins Pazzis, \emph{The classification of large spaces of matrices with bounded rank},
\newblock{preprint,}
\newblock{2013,}
\newblock{arXiv: http://arxiv.org/abs/1004.0298}


\bibitem{dSPlargelinpres}
C. de Seguins Pazzis, The linear preservers of non-singularity in a large space of matrices,
\newblock{\em  Linear Algebra Appl.}
\newblock{\textbf{436-9}}
\newblock{(2012)}
\newblock{3507-3530.}


\end{thebibliography}
\end{document}